\documentclass[11pt,a4paper]{article}

\usepackage{mathrsfs} 

\usepackage{graphicx}     
\usepackage[a4paper, hmargin={2.5cm,2.5cm},vmargin={3.3cm,3.3cm}]{geometry}
\usepackage{amsmath, amssymb, amsfonts, amsbsy, latexsym, color,dsfont}
\usepackage{amsthm}
\usepackage[T1]{fontenc}
\usepackage[latin1]{inputenc}
\usepackage[english]{babel}

\usepackage{cite}
\usepackage{paralist} 
\usepackage{url}
\usepackage[bf,compact,pagestyles,small]{titlesec}
\titlespacing*{\section}{0pt}{14pt}{4pt}
\titlespacing*{\subsection}{0pt}{8pt}{3pt}
\usepackage{mathdots}
\usepackage{tikz-cd}
\usetikzlibrary{matrix,arrows,decorations.pathmorphing}
\usetikzlibrary{decorations.markings}
\tikzset{
  notarrow/.style = {decoration = {markings, mark = at position 0.5 with { \node[transform shape, xscale = 0.9, yscale=0.9,font={\sf}] {X}; } }, postaction = {decorate} },
}

\usepackage{fancyhdr,lastpage}
\def\maketimestamp{\count255=\time
\divide\count255 by 60\relax
\edef\thetime{\the\count255:}%
\multiply\count255 by-60\relax
\advance\count255 by\time
\edef\thetime{\thetime\ifnum\count255<10 0\fi\the\count255}
\edef\thedate{\number\day-\ifcase\month\or Jan\or Feb\or Mar\or
             Apr\or May\or Jun\or Jul\or Aug\or Sep\or Oct\or
             Nov\or Dec\fi-\number\year}
\def\timstamp{\hbox to\hsize{\tt\hfil\thedate\hfil\thetime\hfil}}}
\maketimestamp
\fancypagestyle{plain}{%
\fancyhf{} 
\fancyfoot[C]{\small \thepage{} of \pageref{LastPage}}}

\numberwithin{equation}{section}  
\allowdisplaybreaks[4]

\newtheorem{theorem}{Theorem}[section]

\newtheorem{lemma}[theorem]{Lemma}
\newtheorem{proposition}[theorem]{Proposition}
\newtheorem{corollary}[theorem]{Corollary}
\theoremstyle{definition}
\newtheorem{definition}[theorem]{Definition} 
\newtheorem{example}{Example}
\theoremstyle{remark}
\newtheorem{remark}{Remark}

\DeclareMathOperator{\Span}{span} %
\DeclareMathOperator{\supp}{supp} %
\DeclareMathOperator*{\esssup}{ess\,sup} %
\DeclareMathOperator*{\essinf}{ess\,inf} %

\DeclareMathOperator{\exponential}{e}

\DeclareMathOperator{\im}{im}

\newcommand{\tiS}[1][g]{\ensuremath{\set{T_h {#1}_p}_{\hvar \in \LH, p \in P}}} 
\newcommand{\gaborset}[1][g]{\ensuremath{\set{E_{\gamma}T_{\lambda} {#1}}_{\lambda\in\Lambda,\gamma\in\Gamma}}} 
\newcommand{\gabor}[1][g]{\ensuremath{E_{\gamma}T_{\lambda} {#1}}} 
\newcommand{\SynthesisOperator}{\textnormal{F}}
\newcommand{\anni}{\mathsf{A}}

\newcommand{\myexp}[1]{\exponential^{#1}}

\newcommand{\restrict}[3][]{\ensuremath{#2\rvert_{#3}^{#1}}} 

\newcommand{\card}[1]{\abs{#1}} 

\newcommand*{\numbersys}[1]{\ensuremath{\mathbb{#1}}}
\newcommand*{\C}{\numbersys{C}}
\newcommand*{\R}{\numbersys{R}}

\newcommand*{\Q}{\numbersys{Q}}

\newcommand*{\Z}{\numbersys{Z}}

\newcommand*{\N}{\numbersys{N}}
\newcommand*{\T}{\numbersys{T}}

\newcommand*{\cH}{\mathcal{H}}

\newcommand*{\cF}{\mathcal{F}}
\newcommand*{\cT}{\mathcal{T}}
\newcommand*{\cG}{\mathcal{G}}
\newcommand*{\cK}{\mathcal{K}}

\newcommand{\abs}[1]{\ensuremath{\left\lvert#1\right\rvert}}
\newcommand{\abssmall}[1]{\ensuremath{\lvert#1\rvert}}
\newcommand{\absbig}[1]{\ensuremath{\bigl\lvert#1\bigr\rvert}}
\newcommand{\absBig}[1]{\ensuremath{\Bigl\lvert#1\Bigr\rvert}}
\newcommand{\norm}[2][]{\ensuremath{\left\lVert#2\right\rVert_{#1}}}

\newcommand{\normsmall}[2][]{\ensuremath{\lVert#2\rVert_{#1}}}
\newcommand{\innerprod}[3][]{\ensuremath{\left\langle #2,#3\right\rangle_{\! #1}}}

\newcommand{\set}[1]{\ensuremath{\left\lbrace{#1}\right\rbrace}}
\newcommand{\seq}[1]{\ensuremath{\left\lbrace{#1}\right\rbrace}}
\newcommand{\setprop}[2]{\ensuremath{\left\lbrace{#1} : {#2}\right\rbrace}}

\newcommand{\seqsmall}[1]{\ensuremath{\lbrace{#1}\rbrace}}

\newcommand{\lat}[1]{\ensuremath {#1}} 
\newcommand{\hvar}{h}
\newcommand{\LG}{\ensuremath\lat{\Gamma}}
\newcommand{\LL}{\ensuremath\lat{\Lambda}}
\newcommand{\LH}{\ensuremath\lat{H}}
\newcommand\cD{\mathcal{D}} 

\usepackage{xspace}
\newcommand{\ie}{i.e.,\xspace} 
\newcommand{\eg}{e.g.,\xspace} 

\newcommand{\etc}{etc.\@\xspace}
\newcommand{\cf}{cf.\xspace}

\newcommand{\almoste}{\text{a.e.}}

\newcommand{\cfm}{\mu_M} 
\newlength{\dhatheight}
\newcommand{\doublehat}[1]{%
	\settoheight{\dhatheight}{\ensuremath{\widehat{#1}}}%
	\addtolength{\dhatheight}{-0.35ex}%
    	\widehat{\vphantom{\rule{1pt}{\dhatheight}}%
    	\smash{\hspace{-2pt} \widehat{#1}}}}

\newcommand{\ghat}{\widehat{G}}
\newcommand{\ghhat}{\doublehat{G}}
\newcommand*\oline[1]{%
  \vbox{%
    \hrule height 0.5pt
    \kern0.25ex
    \hbox{%
      \kern-0.1em
      \ifmmode#1\else\ensuremath{#1}\fi
      \kern-0.1em
    }
  }
}

{\nopagebreak\par\noindent\hbox to \hsize{\hrulefill}\par\noindent}

\usepackage{hyperref}
\hypersetup{
 pdfview={FitH},
 pdfstartview={FitH},
 pdfauthor = {Mads Sielemann Jakobsen, Jakob Lemvig},
 pdftitle = {Co-compact Gabor systems on LCA groups}, 
 pdfsubject = {harmonic analysis},
 pdfkeywords = {Gabor systems, characterizing equations, dual frames},
 pdfcreator = {LaTeX with hyperref package},
 pdfproducer = PDFlatex}

\makeatletter
\def\blfootnote{\xdef\@thefnmark{}\@footnotetext} 
\def\subjclass{\xdef\@thefnmark{}\@footnotetext}
\long\def\symbolfootnote[#1]#2{\begingroup%
\def\thefootnote{\fnsymbol{footnote}}\footnote[#1]{#2}\endgroup} 
\if@titlepage
  \renewenvironment{abstract}{%
      \titlepage
      \null\vfil
      \@beginparpenalty\@lowpenalty
      \begin{center}%
        \bfseries \abstractname
        \@endparpenalty\@M
      \end{center}}%
     {\par\vfil\null\endtitlepage}
\else
  \renewenvironment{abstract}{%
      \if@twocolumn
        \section*{\abstractname}%
      \else
        \small
        \list{}{%
          \settowidth{\labelwidth}{\textbf{\abstractname:}}
          \setlength{\leftmargin}{50pt}
          \setlength{\rightmargin}{50pt}
          \setlength{\itemindent}{\labelwidth}
          \addtolength{\itemindent}{\labelsep}
        }
        \item[\textbf{\abstractname:}]

      \fi}
      {\if@twocolumn\else\endlist\fi}
\fi
\makeatother

\begin{document}

\title{Co-compact Gabor systems on locally compact abelian groups}

\date{April 15, 2015}

 \author{Mads Sielemann Jakobsen\footnote{Technical University of Denmark, Department of Applied Mathematics and Computer Science, Matematiktorvet 303B, 2800 Kgs.\ Lyngby, Denmark, E-mail: \protect\url{msja@dtu.dk}}\phantom{$\ast$}, Jakob Lemvig\footnote{Technical University of Denmark, Department of Applied Mathematics and Computer Science, Matematiktorvet 303B, 2800 Kgs.\ Lyngby, Denmark, E-mail: \protect\url{jakle@dtu.dk}}}

 \blfootnote{2010 {\it Mathematics Subject Classification.} Primary
   42C15, Secondary: 43A32, 43A70.} \blfootnote{{\it Key words
     and phrases.} Dual frames, duality principle, fiberization, frame, Gabor
   system, Janssen representation, LCA group, Walnut representation,
   Wexler-Raz biorthogonality relations, Zak transform}

\maketitle 

\thispagestyle{plain}
\begin{abstract} 
In this work we extend classical structure and duality results in Gabor analysis
  on the euclidean space to the setting of second countable locally
  compact abelian (LCA) groups.  We formulate the concept of
  rationally oversampling of Gabor systems in an LCA group and prove
  corresponding characterization results via the Zak transform. From
  these results we derive non-existence results for critically sampled
  continuous Gabor frames. We obtain general characterizations in time
  and in frequency domain of when two Gabor generators yield dual
  frames. Moreover, we prove the Walnut and Janssen representation of the
  Gabor frame operator and consider the Wexler-Raz biorthogonality
  relations for dual generators. Finally, we prove 
  the duality principle for Gabor frames.  Unlike most duality results
  on Gabor systems, we do not rely on the fact that the translation
  and modulation groups are discrete and co-compact subgroups.  Our results only
  rely on the assumption that either one of the translation and
  modulation group (in some cases both) are co-compact subgroups of
  the time and frequency domain. This presentation offers a unified approach
  to the study of continuous and the discrete Gabor frames.
\end{abstract}

\section{Introduction}
\label{sec:introduction}

In Gabor analysis structure and duality results, such as the Zibulski-Zeevi, the Walnut
and the Janssen representation of the frame operator, the Wexler-Raz
biorthogonal relations, and the duality principle, play an important
role. These results go back to a
series of papers in the 1990s
\cite{MR1350701,MR1310658,MR1350700,TolOrr92,
  MR1353539,MR1155734,ISI:A1990EJ96700001,18290047,MR1601115,MR1460623,MR1350650,MR1448221}
on (discrete) regular Gabor systems in $L^2(\R)$ and $L^2(\R^n)$ with modulations and translations along
full-rank lattices. The results now constitute a fundamental part of the
theory. In $L^2(\R^n)$, a regular Gabor system is a discrete family
of functions of the form $\set{E_\gamma T_\lambda g}_{\lambda \in A\Z^n,\gamma \in
  B\Z^n}$, where $g \in L^2(\R^n)$, $E_\gamma
T_\lambda g(x)=\myexp{2\pi i \gamma \cdot x } g(x-\lambda)$, and $A,B
\in {\text{GL}}_n(\R)$.

For Gabor systems on locally compact abelian (LCA) groups, the picture
is a lot less complete.  Rieffel~\cite{MR941652} proved in 1988 a weak
form of the Janssen representation called the \emph{fundamental
  identity in Gabor analysis} (FIGA) for Gabor systems in $L^2(G)$
with modulations and translations along a closed subgroup in $G \times
\ghat$, where $G$ is a second countable LCA group and $\ghat$ its dual
group. Most other structure and duality results assume Gabor systems in
$L^2(G)$ with modulations and translations along discrete and co-compact subgroups (also called uniform lattices), \eg
the Wexler-Raz biorthogonal relations for such uniform lattice Gabor systems
appear implicitly in the work of Gr\"o{}chenig~\cite{MR1601095}. Uniform
lattices are discrete subgroups whose quotient group is compact, and
thus, they are natural generalizations of the concept of full-rank
lattices in $\R^n$. However, not all LCA groups possess uniform
lattices. This naturally leads to the
question to what extent the classical results on Gabor theory mentioned above can be formulated for
non-lattice Gabor systems. The current paper gives an answer to this question.

Thus, in this work we set out to extend the theory of structure and
duality results to a large class of 
Gabor systems in $L^2(G)$, where $G$ is a second countable LCA
group. We will focus on so-called \emph{co-compact Gabor systems}
$\gaborset$, where translation and modulation of $g \in L^2(G)$ are
along closed, co-compact (\ie the quotient group is compact) subgroups
$\LL \subset G$ and $\LG \subset \ghat$, respectively. In $L^2(\R^n)$
co-compact Gabor systems are of the form $\set{\myexp{2\pi i \gamma \cdot x
  }g(x-\lambda)}_{\lambda \in A
  (\R^s \times \Z^{d-s}),\gamma \in B (\R^r \times \Z^{d-r}) }$ for some choice of $0\le r,s \le
d$. Depending on the parameters $r$ and $s$, these Gabor systems
range from discrete over semi-continuous to continuous families. If
only one of the subsets $\LL$ and $\LG$ is a closed, co-compact
subgroup, we will use the terminology \emph{semi} co-compact Gabor
system. Clearly, co-compact and semi co-compact Gabor systems need not be
discrete. More importantly, such systems exist for all LCA groups, and
this setup unifies discrete and continuous Gabor theory. 

For co-compact Gabor systems we prove Walnut's representation
(Theorem~\ref{th:Gabor-walnut-dense}) and Janssen's representation
(Theorem~\ref{th:JanssenRepImpr}) of the Gabor frame operator, the
Wexler-Raz biorthogonal relations (Theorem~\ref{th:WexRaz}), and the
duality principle (Theorem~\ref{thm:ron-shen-duality}). As an example,
we mention that this generalized duality principle for $L^2(\R^n)$
says that the co-compact Gabor system
\[
\setprop{\myexp{2\pi i \gamma \cdot x }g(x-\lambda)}{\lambda \in A (\R^s \times \Z^{d-s}),\gamma
  \in B (\R^r \times \Z^{d-r}) }
\]
 is a continuous frame if, and only if,
the adjoint system 
\[\setprop{\myexp{2\pi i \gamma \cdot x }g(x-\lambda)}{\lambda \in (B^T)^{-1}
  (\{0\}^r \times \Z^{d-r}),\gamma \in (A^T)^{-1} (\{0\}^s \times \Z^{d-s}) }\] 
is a Riesz sequence.  We recall that a family of vectors $\seq{f_k}_{k \in M}$
in a Hilbert space $\cH$ is a continuous frame with respect to a measure $\mu$
on the index set $M$ if $\norm{f}^2 \asymp \int_{M}\abs{\innerprod{f}{f_k}}^2 d\mu$ for all $f
\in \cH$ and that $\seq{g_k}_{k \in \N}$ is a Riesz sequence if
$\norm[\ell^2]{c}^2 \asymp \norm{\sum_{k}c_k g_k}^2$ for all finite sequences
$c=\seq{c_k}_{k \in \N}$. Our proof of the duality principle relies on a simple
characterization of Riesz sequences in Hilbert spaces (Theorem
\ref{thm:riesz-seq}).

As we will see, the setting of co-compact Gabor systems is indeed a natural framework for 
structure and duality results. Closedness of the modulation and translation
subgroups is a standard assumption, and one cannot get very far
without it, \eg closedness allows for applications of key
identifications between subgroups and their annihilators as well as
applications of the Weil and the Poisson formulas. Co-compactness is, on
the other hand, non-standard, and to the best of our knowledge this
work is the first systematic study of co-compact Gabor systems.  Under
the second countability assumption on $G$, co-compactness is the weakest
assumption that yields an \emph{adjoint} Gabor system with modulations
and translations along discrete and countable subgroups. In
this way, co-compactness of the respective subgroups is the most
general setting for which the Wexler-Raz biorthogonal relations, the
duality conditions for dual generators and the duality principle can
be phrased in a way that resembles the classical statements in $L^2(\R^n)$.  As an example we
mention that in the Wexler-Raz biorthogonal relations, one
characterize duality of two Gabor frames by a biorthogonality
condition of the corresponding adjoint Gabor systems. Since $L^2(G)$
is separable, such a biorthogonality condition is only possible if the
adjoint systems are \emph{countable} sequences (which co-compactness
exactly guarantees). Furthermore, co-compact Gabor systems are
precisely the setting, where the Walnut and Janssen representation of
the continuous frame operator are a \emph{discrete} representation.

However, we begin our work on Gabor systems with a study of \emph{semi}
co-compact Gabor systems as special cases of co-compact translation
invariant systems, recently introduced in
\cite{BowRos2014,JakobsenReproducing2014}. For translation invariant
systems we consider fiberization characterization of frames for
translation invariant subspaces
(Theorem~\ref{thm:fibers-characterization}), generalizing results from
\cite{BowRos2014,MR2578463,MR1350650,MR1795633}. Using these
fiberization techniques we will develop Zak transform methods for
Gabor analysis in $L^2(G)$. This leads among other things to a concept
of rational oversampling in LCA groups
(Theorem~\ref{thm:char-vector-valued-Zak-time}) and a Zibulski-Zeevi
representation (Corollary~\ref{thm:char-Zibulski-Zeevi}). Furthermore,
we will prove the non-existence of continuous, semi co-compact Gabor
frames at ``critical density''
(Theorem~\ref{thm:non-existence-cont-gabor-critical-sampling}). We
also give characterizations of generators of dual semi co-compact
Gabor frames (Theorems~\ref{th:dual-Gabor-time} and
\ref{th:dual-Gabor}).

There are several advantages of the LCA group approach, one being that
the essential ingredient in our arguments often becomes more
transparent than in the special cases. The abstract approach also
allows us to unify results from the standard settings where $G$ is
usually $\R^n$, $\Z^n$, or $\Z_n$. This is not only useful for the
sake of generalizations, but, in some instances, it can also simplify
the proofs in the special cases. As an example we mention that our proof
of the Zak transform characterization of Gabor frames is based on two
applications of the same result on fiberizations of $L^2(G)$, but for
two different LCA groups $G$. In the Euclidean setting this would
require two different fiberization results, one for $G=\R^n$ and one
for $G=A\Z^n$ for $A \in \text{GL}_n(\R)$. In the setting of LCA groups we
can unify such results into one general result. On the other hand,
even for $G=\R^n$ most of our results are new. 

For related work  on
locally compact (abelian) groups we refer to the recent
papers \cite{BalazsSpeckbacher2014,BowRos2014,
  MR2578463,JakobsenReproducing2014,MR2283810,OleSaySon2014,MR2490216,MR3019110}
as well as the book \cite{MR1601119} 
and the references therein. 

The paper is organized as follows. In Section~\ref{sec:preliminaries}
we give a brief introduction to harmonic analysis on LCA groups and
frame theory. In Section~\ref{sec:TI-systems} we study co-compact
translation invariant systems, and specialize to \emph{semi} co-compact
Gabor systems in Section~\ref{sec:semi-co-compact}.  In
Section~\ref{sec:Gabor-frame-operator} we study the frame operator of
Gabor systems, and in Section~\ref{sec:duality-results} we present
duality results on co-compact Gabor frames. 

\section{Preliminaries} 
\label{sec:preliminaries}

In the following sections we set up notation and recall some useful results from Fourier analysis on
locally compact abelian groups and continuous frame theory.
\subsection{Fourier analysis on locally compact abelian groups} \label{sec:LCA}

In this paper $G$ will denote a second countable locally compact abelian group. To $G$ we associate
its dual group $\ghat$ which consists of all characters, \ie all continuous homomorphisms from $G$
into the torus $\T \cong \setprop{z\in\C}{ \abs{z} =1}$. Under pointwise multiplication $\ghat$ is
also a locally compact abelian group. Throughout the paper we use addition and multiplication as
group operation in $G$ and $\ghat$, respectively. By the Pontryagin duality theorem, the dual group
of $\ghat$ is isomorphic to $G$ as a topological group, \ie $\ghhat \cong G$. Moreover, if $G$ is
discrete, then $\ghat$ is compact, and if $G$ is compact, then $\ghat$ is discrete.

We denote the Haar measure on $G$ by $\mu_G$. The (left) Haar measure on any locally compact group
is unique up to a positive constant. From $\mu_G$ we define $L^1(G)$ and the Hilbert space $L^2(G)$
over the complex field in the usual way. $L^2(G)$ is separable, because $G$ is assumed to be second
countable. For functions $f\in L^1(G)$ we define the Fourier transform
\[ 
\cF f(\omega) = \hat{f}(\omega) = \int_{G} f(x) \, \overline{\omega(x)} \, d\mu_G(x), \quad
\omega \in \ghat.
\] 
If $f\in L^1(G), \hat{f} \in L^1(\ghat)$, and the measure on $G$ and $\ghat$ are normalized so that
the Plancherel theorem holds (see \cite[(31.1)]{MR0262773}), the function $f$ can be recovered from
$\hat{f}$ by the inverse Fourier transform
\[
f(x) = \cF^{-1}\hat{f}(x) = \int_{\ghat} \hat{f}(\omega) \, \omega(x) \, d\mu_{\ghat}(\omega),
\quad a.e.\ x\in G.
\]
We assume that the measure on a group $\mu_G$ and its dual group $\mu_{\ghat}$ are normalized this
way, and we refer to them as \emph{dual measures}. We will consider $\cF$ as an isometric isomorphism
between $L^2(G)$ and $L^2(\ghat)$.

On any locally compact abelian group $G$, we define the following three operators. For $a \in G$, the
operator $T_{a}$, called \emph{translation} by $a$, is defined by
\[
 T_{a}:L^2(G)\to L^2(G), \ (T_{a}f)(x) = f(x-a), \quad x\in G.
\]
For $\chi\in \ghat$, the
operator $E_{\chi}$, called \emph{modulation} by $\chi$, is defined by
\[ 
E_{\chi}:L^2(G)\to L^2(G), \ (E_{\chi}f)(x) = \chi(x) f(x), \quad x\in G.
\]
For $t \in L^\infty(G)$ the operator $M_{t}$, called \emph{multiplication} by $t$, is defined by
\[
M_{t}:L^2(G)\to L^2(G), \ (M_{t}f)(x) = t(x) f(x), \quad x\in G.
\] 
The following commutator relations will be used repeatedly: $T_aE_{\chi} = \overline{\chi(a)}\,
E_{\chi}T_a , \ \cF T_a = E_{a^{-1}} \cF$, and $ \cF E_{\chi} = T_{\chi} \cF.$

For a subset $H$ of an LCA group $G$, we define its annihilator as
\[ 
\anni(\ghat, H) = \{ \omega \in \ghat \, | \, \omega(x) = 1 \ \text{for all} \ x\in H \}.
\] 
When the group $\ghat$ is understood from the context, we will simply denote the annihilator
$\anni(\ghat, H) =H^{\perp}$. The annihilator is a closed subgroup in $\ghat$, and if $H$ is a
closed subgroup itself, then $\widehat H \cong \ghat / H^{\perp}$ and $\widehat{G/H} \cong
H^{\perp}$. These relations show that for a closed subgroup $H$ the quotient $G/H$ is compact if and
only if $H^{\perp}$ is discrete.

\begin{lemma} \label{le:finite-quotient-cong-annihilator} Let $H$ be a closed subgroup of $G$. If $G/H$ is finite, then $H^{\perp} \cong G/H$.
\end{lemma}
\begin{proof} Note that any finite group $G$ is self-dual, that is, $\ghat \cong G$. And so, by
  application of the isomorphism $H^{\perp} \cong \widehat{G/H}$ we find that $  H^{\perp} \cong \widehat{G/H} \cong G/H$.
\end{proof}

We also remind the reader of Weil's formula; it relates integrable functions over $G$ with
integrable functions on the quotient space $G/H$ when $H$ is a closed normal subgroup of $G$. For a
closed subgroup $H$ of $G$ we Let $\pi_H:G\to G/H, \ \pi_H(x) = x+H$ be the \textit{canonical map}
from $G$ onto $G/H$. If $f\in L^1(G)$, then the function $\dot x \mapsto \int_H f(x+h) \, d\mu_{H}(h)$, $\dot
x = \pi_H(x)$ defined almost everywhere on $G/H$, is integrable. Furthermore, when two of the Haar
measures on $G, H$ and $G/H$ are given, then the third can be normalized such that
\begin{equation}
\label{eq:1802a}  \int_G f(x) \, dx = \int_{G/H} \int_H f(x+h) \, d\mu_{H}(h) \, d\mu_{G/H}(\dot x) .
\end{equation}
%
Hence, if two of the measures on $G,H,G/H,\ghat, H^{\perp}$ and $\widehat G / H^{\perp}$ are given, and these two are not dual measures, then by
requiring dual measures and Weil's formula \eqref{eq:1802a}, all other
measures are uniquely determined. To ease notation, we will often write $dh$ in place of $d\mu_H(h)$ and likewise for other
measures.

A \emph{Borel section} or a fundamental domain of a closed subgroup $H$ in $G$ is
a Borel measurable subset $X$ of $G$ which meets each coset $G/H$ once. Any
closed subgroup $H$ in $G$ has a Borel section \cite[Lemma 1.1]{MR0044536};
however, we shall in the following usually only consider Borel sections of discrete
subgroups $H$. We always equip Borel sections of $G$ with the Haar measure $\restrict{\mu_G}{X}$. Assume that $H$ is a discrete subgroup.
It follows that $\mu_{G}(X)$ is finite if, and only if,  $H$ is co-compact, \ie
$H$ is a uniform lattice \cite{BowRos2014}.  From \cite{BowRos2014}, we also
have that the mapping $x \mapsto x + H$ from $(X,\mu_G)$ to $(G/H,\mu_{G/H})$ is
measure-preserving, and the mapping $Q(f) = f'$ defined by
\begin{equation}
  \label{eq:isom-borel}
  f'(x+H) = f(x), \quad x+H \in G/H, x \in X,
\end{equation}
is an isometry from $L^2(X,\mu_G)$ onto $L^2(G/H,\mu_{G/H})$.

For more information on harmonic analysis on locally compact abelian groups, we refer the reader to
the classical books \cite{MR1397028,MR0156915,MR0262773,MR1802924}.

\subsection{Frame theory}
\label{sec:frame-theory}
One of the central concept of this paper is that of a frame. The definition is as follows.
\begin{definition} \label{def:cont-frames} Let $\cH$ be a complex Hilbert space, and
  let $(M,\Sigma_M,\cfm)$ be a measure space, where $\Sigma_M$ denotes the $\sigma$-algebra and
  $\cfm$ the non-negative measure. A family of vectors $\set{f_k}_{k \in M}$ is called a
  \emph{frame} for $\cH$ with respect to $(M,\Sigma_M,\cfm)$ if
  \begin{enumerate}[(a)]
  \item 
    the mapping $M \to \C, k \mapsto \langle f,f_k\rangle $ is
    measurable for all $f \in \cH$, and
  \item there exists constants $A,B>0$ such that
    \begin{equation}
      \label{eq:cont-frame-inequality} 
      A \norm{f}^2 \le \int_M \abs{\innerprod{f}{f_k}}^2 d\cfm(k) 
      \le B \norm{f}^2 \quad \text{for all } f \in \cH. 
    \end{equation}
  \end{enumerate}
The constants $A$ and $B$ are called \emph{frame bounds}.
\end{definition}
If $\set{f_k}_{k\in M}$ is measurable and the upper bound in the above
inequality~(\ref{eq:cont-frame-inequality}) holds, then $\set{f_k}_{k\in M}$ is said to be a
\emph{Bessel} system or family with constant $B$. A frame $\set{f_k}_{k\in M}$ is said to be \emph{tight} if
we can choose $A = B$; if, furthermore, $A = B = 1$, then $\set{f_k}_{k\in M}$ is said to be a
\emph{Parseval frame}.  

If $\cfm$ is the counting measure and $\Sigma_M=2^M$ the discrete
$\sigma$-algebra, we say that $\set{f_k}_{k\in M}$ is a \emph{discrete frame}
whenever (\ref{eq:cont-frame-inequality}) is satisfied; for this measure space,
any family of vectors is obviously measurable. \emph{Because the results of the
  present paper can be formulated for the discrete and continuous setting, we
  shall refer to either cases as \emph{frames} and be more specific when
  necessary}. We mention that in the literature frames and discrete frames are
usually called continuous frames and frames, respectively. The concept of continuous frames was introduced by Kaiser\cite{MR1287849} and Ali, Antoine, and Gazeau~\cite{MR1206084}. For an introduction to frame theory, we refer the reader to \cite{MR2428338}.

To a Bessel family $\set{f_k}_{k\in M}$ for $\cH$, we associate the the \emph{synthesis operator}
$T: L^2(M,\cfm) \to \cH$ defined weakly by
\begin{equation}
T \set{c_k}_{k \in M} = \int_{M} c_k f_k \, \cfm(k).\label{eq:synthesis-op}
\end{equation}
This is a bounded linear operator. Its adjoint operator $T^\ast: \cH \to L^2(M,\cfm)$ is called the
\emph{analysis operator}, and it is given by
\begin{equation}
T^\ast f = \set{\innerprod{f}{f_k}}_{k \in M}.\label{eq:analysis-op}
\end{equation}
The \emph{frame operator} $S : \cH \to \cH$ is then defined as $S=TT^\ast$. We remark that the frame
operator is the unique operator satisfying
\begin{equation}
\innerprod{S f }{g} = \int_{M} \innerprod{f}{f_k}\innerprod{f_k}{g} d\cfm(k) \quad \, \text{for all } f,g \in \cH,
\label{eq:def-frame-op-in-H}
\end{equation}
and that it is well-defined, bounded and self-adjoint for any Bessel system $\set{f_k}_{k \in M}$;
it is invertible if $\set{f_k}_{k \in M}$ is a frame.

In case the frame inequalities (\ref{eq:cont-frame-inequality}) only hold for $f
\in \cK := \overline{\Span} \seq{f_k}_{k \in M} \subset \cH$, we say that
$\seq{f_k}_{k \in M}$ is a \emph{basic frame} or a frame for its closed linear
span. For discrete frames such frames are usually called frame sequences;
we will not adopt this terminology as basic frames need not be sequences. A frame for $\cH$ is clearly a basic frame with $\cK=\cH$. If we
need to stress that a basic frame spans all of $\cH$, we use the terminology
\emph{total frame}. Now, let us briefly comment on the definition of the subspace $\cK$. 

From the Bessel property of a (basic) frame $\seq{f_k}$, we see that:
\[ 
\overline{\im T} = (\ker T^\ast)^\perp = \setprop{f \in
  \cH}{\innerprod{f}{f_k}=0 \ \forall k \in M}^\perp = \overline{\Span}
\seq{f_k}_{k \in M}. 
\] 
The lower frame bound for $f \in \cK$ implies that the operator $\restrict{T^\ast}{\cK}$ is
bounded from below, i.e., $\norm{\restrict{T^\ast}{\cK}f} \ge \sqrt{A}
\norm{f}$, which is equivalent to $\restrict{T^\ast}{\cK}$ being
injective with closed range which, in turn, implies that $T$ has closed
range. Since $\restrict{T^\ast}{\cK}$ is injective, the range of $T$ is dense in
$\cK$. It follows that $\im T=\cK$. 

We will only consider measures $\mu_M$ that are $\sigma$-finite. Assume that $\seq{f_k}$ is measurable. It is known that $T$ as in (\ref{eq:synthesis-op}) defines a bounded linear operator if, and only if,  $\seq{f_k}_{k \in M}$ is a Bessel family~\cite{MR2238038}. Hence, the argument in the preceding paragraph shows that $\seq{f_k}_{k \in M}$ is a basic frame if, and only if,  $T$ as in (\ref{eq:synthesis-op}) defines a bounded linear operator with $\im T = \cK$.

Two Bessel systems $\set{f_k}_{k\in M}$ and $\set{g_k}_{k\in M}$ are said to be \emph{dual frames} for $\cH$
if
\begin{equation}
  \label{eq:cont-dual-weak}
  \innerprod{f}{g} = \int_M \innerprod{f}{g_k} \innerprod{f_k}{g} d\cfm (k) \quad \, \text{for all } f,g \in \cH.
\end{equation}
In this case
\begin{equation}
 f = \int_M \innerprod{f}{g_k}f_k \, d\cfm (k) \quad \, \text{for $f \in \cH$}, \label{eq:frame-rep-weak-sense}
\end{equation}
holds in the weak sense. For discrete frames, equation~(\ref{eq:frame-rep-weak-sense}) holds in the
usual strong sense, \ie with (unconditional) convergence in the $\cH$ norm. Two dual frames are
indeed frames. We also mention that to a given frame for $\cH$ one can always find at least one dual
frame, the so-called canonical dual frame $\{S^{-1}f_k\}_{k \in
  M}$. 

Let us end this section with the definition of a Riesz sequence.
\begin{definition} 
\label{def:riesz-seq}
Let $\{f_k\}_{k=1}^{\infty}$ be a sequence in a Hilbert space $\mathcal H$. If there exists constants $A,B>0$ such that
\[ A \, \sum_{k} \vert c_k \vert^2 \le \Big\Vert  \sum_{k} c_k f_k \, \Big\Vert_{\mathcal H}^2 \le B \, \sum_{k} \vert c_k \vert^2\]
for all finite sequence $\{c_k\}_{k=1}^{\infty}$, then we call $\{f_k\}_{k=1}^{\infty}$ a \emph{Riesz sequence}. If furthermore $\overline{\text{span}} \{f_k\}_{k=1}^{\infty}=\mathcal{H}$, then $\{f_k\}_{k=1}^{\infty}$ is a \emph{Riesz basis}.
\end{definition}

\section{Translation invariant systems}
\label{sec:TI-systems}

Before we focus on Gabor systems, let us first show some results concerning the class of translation
invariant systems, recently introduced in \cite{JakobsenReproducing2014,BowRos2014}, which contains
the class of (semi) co-compact Gabor systems.

We define translation invariant systems as follows.
Let $P$ be a countable or an
uncountable index set, let $g_{p} \in L^2(G)$ for $p \in P$, and let $\LH$ be a closed,
co-compact subgroup in $G$. For a compact abelian group, the group is metrizable if, and only if,
the character group is countable \cite[(24.15)]{MR0156915}. Hence, since $G/\LH$ is
compact and metrizable, the group $\widehat{G/\LH} \cong \LH^{\perp}$ is
discrete and countable. \emph{Unless stated otherwise we equip $\LH^{\perp}$ with the counting
measure and assume a fixed Haar measure $\mu_G$ on $G$.}

The (co-compact) \emph{translation invariant} (TI) system generated by $\{g_{p}\}_{p\in P}$ with
translation along the closed, co-compact subgroup $\LH$ is the family of functions $\tiS$.  We will
use the following standing assumptions on the index set $P$: 
\begin{enumerate}[(I)]
\item $(P,\Sigma_{P},\mu_{P})$ is a $\sigma$-finite measure space, \label{eq:Hypo1}
\item $p \mapsto g_p, (P,\Sigma_{P}) \to (L^2(G),B_{L^2(G)})$ is measurable,
  \label{eq:Hypo2} 
\item $(p,x)\mapsto g_p(x), (P \times G, \Sigma_{P} \otimes B_G) \to (\C, B_\C)$ is
  measurable. \label{eq:Hypo3}
\end{enumerate}
We say that $\{g_{p}\}_{p\in P}$ is \emph{admissible} or, when $g_{p}$
is clear from the context, simply that the measure space $P$ is admissible. The nature of
these assumptions are discussed in
\cite{JakobsenReproducing2014}. Observe that any closed subgroup $P$ of
$G$ (or $\ghat$) with the Haar measure is admissible if $p \to g_p$ is
continuous, e.g., if $g_p=T_p g$ for some function $g\in L^2(G)$. 

If $P$ is countable, we equip it with a weighted counting measure. If the subgroup $\LH$ is also
discrete, hence a uniform lattice, the system $\tiS$ is a \emph{shift} invariant (SI) system.

\subsection{Fiberization}
\label{sec:fiberization}

TI systems are of interest to us since the Gabor systems we shall study are special instances of
these. As the work of Ron and Shen~\cite{MR1350650} and Bownik~\cite{MR1795633} show, certain
Gramian and so-called dual Gramian matrices as well as a fiberization technique play an important
role in the study of TI systems. The fiberization technique is closely related to Zak transform
methods in Gabor analysis, as we will see in Section~\ref{sec:char-equat}.

Let $\Omega \subset \ghat$ be a Borel section of $\LH^\perp$
in $\ghat$ as defined in Section~\ref{sec:LCA}. 
Following \cite{BowRos2014} we define the \emph{fiberization} mapping $\cT : L^2(G) \to
L^2(\Omega,\ell^2(\LH^\perp))$ by
\begin{equation}
\cT f(\omega)=\seqsmall{\hat{f}(\omega\alpha)}_{\alpha \in \LH^\perp}, \quad \omega \in \Omega;
\label{eq:fiber-map}
\end{equation}
the inner product in $L^2(\Omega,\ell^2(\LH^\perp))$ is defined in the obvious manner. Fiberization is
an isometric, isomorphic operation as shown in \cite{MR2578463,BowRos2014}.

Our first result characterizes the frame/Bessel property of TI systems in terms of
fibers. It extends results from \cite{BowRos2014,MR2578463,MR1795633} to the case of
uncountable many generators $\set{g_p}_{p \in P}$.
\begin{theorem}
\label{thm:fibers-characterization}
Let $0<A\le B < \infty$, let $\LH \subset G$ be a closed, co-compact subgroup,
and let $\set{g_p}_{p \in P} \subset L^2(G)$, where $(P,\mu_P)$ is an admissible
measure space. The following assertions are equivalent:
  \begin{enumerate}[(i)]
  \item The family $\tiS$ is a frame for $L^2(G)$ with bounds $A$ and $B$ (or a
    Bessel system with bound $B$), \label{item:1}
  \item For almost every $\omega \in \Omega$, the family $\seq{\cT g_p(\omega)}_{p \in P}$ is a frame for $\ell^2(\LH^\perp)$ with bounds $A$ and $B$ (or a Bessel system with
    bound $B$). \label{item:2}
  \end{enumerate}
\end{theorem}
\begin{proof}
 The proof follows from the proofs in
  \cite{BowRos2014,MR2578463,MR1795633}. Indeed, the key computation in \cite{BowRos2014}
  shows that
\[
\int_P \int_H \abs{\innerprod[L^2]{f}{T_hg_p}}^2 d\mu_H(h) d\mu_P(p) = \int_P \int_\Omega
\abs{\innerprod[\ell^2]{\cT f(\omega)}{\cT g_p(\omega)}}^2 d\mu_{\ghat}(\omega) d\mu_P(p)
\]
for all $f \in L^2(G)$. Let us outline the argument for the frame
case; the Bessel case is similar. Assume that (ii) holds. Then for a.e.\ $\omega \in
\Omega$ we have 
\[ A \norm[\ell^2]{a} \le \int_P \abs{\innerprod[\ell^2]{a}{\cT
    g_p(\omega)}}^2 d\mu_P(p) \le B \norm[\ell^2]{a} \quad \text{for
  all } a \in \ell^2(H^\perp). \]
If we integrate these inequalities over $\Omega$ and use that $\cT$ is
an isometric isomorphism, we arrive at (i) using the key computation above. The other implication
follows as in \cite{MR1795633}. 
\end{proof}

 \begin{remark}
   Theorem~\ref{thm:fibers-characterization} can also be formulated
   for basic frames using the notion of range functions. A very
   general version of this result was obtained independently and
   concurrently in \cite{Iverson2014}. Theorem~\ref{thm:fibers-characterization}
   is closely related to the theory of translation invariant subspaces which 
   very recently has been studied in 
\cite{BarbieriHernPater2014,Iverson2014} using Zak transform methods (cf.\ Section~\ref{sec:char-equat}). 
 \end{remark}

Theorem~\ref{thm:fibers-characterization} shows that the task of verifying that
a given TI system $\tiS$ is a frame for $L^2(G)$ can be replaced by the simpler
task of proving that the fibers $\seq{\cT g_p(\omega)}_{p \in P}$ are a frame for
the discrete space $\ell^2(\LH^\perp)$, however, this needs to be done for every
$\omega \in \Omega$. For a uniform lattice $H$, the Borel section $\Omega$ of $H^\perp$ is
compact, but for non-discrete, co-compact closed subgroups $H$, this is not the
case, in fact, $m_{\ghat}(\Omega)=\infty$.

Let $\omega \in \Omega$ be given. The analysis operator
$L_\omega:\ell^2(\LH^\perp)\to L^2(P)$ for the family of fibers $\seq{\cT
  g_p(\omega)}_{p \in P}$ in $\ell^2(\LH^\perp)$ is given by:
\begin{equation}
  \label{eq:analysis-op-fibers}
  L_{\omega} c = p \mapsto \innerprod[\ell^2(\LH^\perp)]{c}{\cT g_p(\omega)}, \quad D(L_\omega)=c_{00}(\LH^\perp). 
\end{equation}
Note that we have only defined the analysis operator $L_\omega$ for finite
sequences since we do not, a priori, assume that the family of fibers is a Bessel
system, cf.\ (\ref{eq:analysis-op}). If $L_\omega$ is bounded, it extends to a
bounded, linear operator on all of $\ell^2(\LH^\perp)$; clearly, $L_\omega$ is
bounded with bound $\norm{L_\omega} \le \sqrt{B}$ if, and only if,  $\seq{\cT
  g_p(\omega)}_{p \in P}$ is a Bessel system with bound $B$. In this case the
adjoint is the synthesis operator $L_\omega^\ast: L^2(P)\to\ell^2(\LH^\perp)$
given by:
\[
L_\omega^\ast f = \seq{\int_P f(p)\, \hat g_p(\omega\alpha) d \mu_P(p)}_{\alpha \in \LH^\perp}, \qquad \text{where } f \in L^2(P). 
\]
From results in \cite[Chapter~3]{MR1946982} and \cite{MR2238038} we know that
this synthesis operator $L_\omega^\ast: L^2(P)\to\ell^2(\LH^\perp)$ is a
well-defined, bounded linear operator if, and only if,  the fibers $\seq{\cT
  g_p(\omega)}_{p \in P}$ is a Bessel system. The frame operator $L_\omega^\ast
L_\omega$ of the family of fibers is called the \emph{dual Gramian} and is denoted by
$\tilde{\cG}_\omega:\ell^2(\LH^\perp) \to \ell^2(\LH^\perp)$. Again, using
results from \cite[Chapter 3]{MR1946982}, the frame operator is a bounded,
linear operator acting on all of $\ell^2(\LH^\perp)$ precisely when the fibers
form a Bessel system. Paying attention to the operator bounds and Bessel
constants, we therefore have the following result, extending results from \cite{MR2578463,BowRos2014} to the case of uncountably many generators.
\begin{proposition}
\label{thm:TI-gramian-pre-bessel-equi}
Let $B>0$, let $\LH \subset G$ be a closed, co-compact subgroup, and let
$\set{g_p}_{p \in P} \subset L^2(G)$, where $(P,\mu_P)$ is an admissible measure
space. The following assertions are equivalent:
  \begin{enumerate}[(i)]
  \item $\tiS$ is a Bessel system with bound $B$,
  \item $\esssup_{\omega \in \Omega} \normsmall{\tilde{\cG}_\omega} \le B$,\label{item:3}
  \item $\esssup_{\omega \in \Omega} \norm{L_\omega} \le \sqrt{B}$.\label{item:4}
  \end{enumerate}
\end{proposition}

In a similar fashion, it is possible to generalize \cite[Proposition
4.9(2)]{MR2578463} and the corresponding result in \cite{BowRos2014} to the case of uncountably many generators.

\section{Semi co-compact Gabor systems and characterizations}
\label{sec:semi-co-compact}

In the the rest of this article we will concentrate on Gabor systems. A \emph{Gabor system} in
$L^2(G)$ with generator $g\in L^2(G)$ is a family of functions of the form
\[
\set{E_{\gamma} T_{\lambda} g}_{\gamma \in \LG, \lambda \in \LL}, \text{where
  $\LG\subseteq \ghat$ and $\LL\subseteq G$}.
\] 
We will usually assume that at least one of the subsets $\LG\subset \ghat$ or $\LL\subset G$ is a
closed subgroup; if either of these subsets is not a closed subgroup, it will be assumed to be, at
least, admissible as an
index set (cf. the previous section). 
We often use that semi co-compact Gabor systems are unitarily equivalent to co-compact translation
 invariant systems in either time or in frequency domain. If both $\Gamma$ and $\Lambda$ are closed and co-compact
subgroups, we say that $\set{E_{\gamma} T_{\lambda} g}_{\gamma \in \LG, \lambda \in \LL}$ is a
\emph{co-compact Gabor system}; if only one of the sets $\Gamma$ and $\Lambda$ is a closed and
co-compact subgroup, we name the Gabor system \emph{semi co-compact}. If both $\Gamma$ and $\Lambda$
are discrete and co-compact, we recover the well-known uniform lattice
Gabor systems. 




\subsection{Characterizations of Gabor frames and the Zak transform} 
\label{sec:char-equat}

The fiberization technique from
Theorem~\ref{thm:fibers-characterization} will play a crucial
role in the characterizations of semi co-compact Gabor frames, presented in this subsection. From
Theorem~\ref{thm:fibers-characterization} for the TI system $\{
T_{\gamma} \cF^{-1} T_{\lambda} g\}_{\gamma \in \LG, \lambda \in
  \LL}$, which is unitarily equivalent with $\set{E_{\gamma} T_{\lambda} g}_{\gamma \in \LG, \lambda \in \LL}$, we immediately have a characterization of the frame property of Gabor
systems. 

\begin{proposition}
\label{thm:char-frame-prop-LG-Gabor}
Let $g \in L^2(G)$, and let $0<A\le B < \infty$. Let $\Gamma$ be a closed, co-compact subgroup of
$\ghat$, and let
$(\Lambda,\Sigma_\Lambda,\mu_\Lambda)$ be an admissible measure space in $G$. The following assertions are
equivalent:
\begin{enumerate}[(i)]
\item $\set{E_{\gamma} T_{\lambda} g}_{\gamma \in \LG, \lambda \in \LL}$ is a frame for $L^2(G)$
  with bounds $A$ and $B$, \label{item:6}
\item $\seq{\seq{g(x+\lambda+\alpha)}_{\alpha \in \LG^\perp}}_{\lambda \in \Lambda}$ is a frame for
  $\ell^2(\LG^\perp)$ with bounds $A$ and $B$ for a.e.\ $x \in X$, where $X$ is a Borel
  section of $\LG^\perp$ in $\ghat$.
\label{item:7}
\end{enumerate}
\end{proposition}


We will apply Theorem~\ref{thm:fibers-characterization} once more to
Proposition~\ref{thm:char-frame-prop-LG-Gabor} under stronger assumptions on
$\LL$. In the following we will always assume that $\LL$ is a closed subgroup of
$G$. For a moment, let us even assume that $\LL=\LG^\perp$, where $\LG$ is a
closed, co-compact subgroup of $\ghat$. Note that this implies that $\LL$ is
discrete and countable. For uniform lattice Gabor systems the condition
$\LL=\LG^\perp$ is called \emph{critical density} by Gr\"o{}chenig
\cite{MR1601095} since Borel sections $X$ and $\Omega$ of the lattices
$\LG^\perp$ and $\LL^\perp$ in this case satisfy $m_G(X) m_{\ghat}(\Omega) =
1$. Theorem 6.5.2 in \cite{MR1601095} states that the uniform lattice Gabor
system $\gaborset$ only can be frame for $L^2(G)$ if $m_G(X) m_{\ghat}(\Omega)
\le 1$. Clearly this is not a necessary condition when either $\LL$ or $\LG$ is
non-discrete since, for closed, co-compact subgroups, a Borel section of its
annihilator has finite measure if and only if the subgroup itself is discrete.

Now, back to the assumption $\LL=\LG^\perp$ with $\LG$ being a (not necessarily
discrete) closed, co-compact subgroup of $\ghat$. In this case, the system in
Proposition~\ref{thm:char-frame-prop-LG-Gabor}\eqref{item:7} is a shift invariant system of the form 
$\seq{T_\lambda \varphi_x}_{\lambda \in \Lambda}$ in $\ell^2(\LL)$ with countably many generators 
$\varphi_x:=\seq{g(x+\alpha)}_{\alpha \in \LL}$.  We now apply the fiberization
techniques from Section~\ref{sec:fiberization} with $G=\LL$ and $\LH=\LL$. Since
the annihilator $\LH^\perp$ in this case is $\anni(\widehat{\LL},\LL)=\{1\}$,
the fiberization map~(\ref{eq:fiber-map}) is simply $\cT
f(\omega)=\{\hat{f}(\omega)\}$ for $\omega \in \Omega$, where $\Omega$ is a
Borel section of $\{1\}$ in $\widehat{\LL}$, hence, $\Omega=\widehat{\LL}$. The
Fourier transform of the generator $\varphi_x \in \ell^2(\LL)$ is
\begin{equation}
\label{eq:def-phi_x}
\hat{\varphi}_x(\omega)=\sum_{\alpha \in \LL}
g(x+\alpha)\overline{\omega(\alpha)},
\end{equation}
which is the \emph{Zak transform} $Z_\LL g(x,\omega)$ of $g$ with respect to the
discrete group $\LL\subset G$. 

By  Theorem~\ref{thm:fibers-characterization} (or a result in
\cite{BowRos2014}, to be more precise), $\seq{T_\lambda
  \varphi_x}_{\lambda \in \Lambda}$ is a basic frame in $\ell^2(\LL)$
with bounds $A$ and $B$ if, and only if, $\{\hat{\varphi}_x(\omega)\}$
is a basic frame in $\ell^2(\anni(\widehat{\LL},\LL))\cong\C$ with
bounds $A,B$ for almost all $\omega \in \widehat{\LL}$. Now, a scalar
$\{\hat{\varphi}_x(\omega)\}$ is a basic frame in $\C$ with bounds $A$
and $B$ if, and only if, its norm squared, whenever non-zero, is
bounded between $A$ and $B$.  We conclude that $\set{E_{\gamma}
  T_{\lambda} g}_{\gamma \in \LL^\perp, \lambda \in \LL}$ is a Gabor
basic frame in $L^2(G)$ with bound $A$ and $B$ if, and only if,
\begin{equation}
\label{eq:1D-Zak-cond}
A \le \absBig{\sum_{\alpha \in \LL} g(x+\alpha)\overline{\omega(\alpha)}\,}^2 \le B
\quad \text{for a.e. } x \in X, \omega \in \Omega= \widehat{\LL} \text{ for which } \hat{\varphi}_x(\omega) \neq 0. 
\end{equation}
In particular, whenever $\LL=\LG^\perp$ with $\LG$ being a closed, co-compact
subgroup, we see that $\set{E_{\gamma} T_{\lambda} g}_{\gamma \in \LL^\perp,
  \lambda \in \LL}$ is a total Gabor frame for all of $L^2(G)$ if, and only if, $A \le
\abssmall{Z_\LL g(x,\omega)}^2 \le B$ for almost any $x \in X, \omega \in
\Omega= \widehat{\LL}$. Still assuming $\LG=\LL^\perp$, this result can be shown
to hold for any closed subgroup $\LL \subset G$ \cite[Theorem 2.6]{MR3019110}. However,
the next result shows a non-existence phenomenon of such continuous Gabor
frames.

\begin{theorem}
  \label{thm:non-existence-cont-gabor-critical-sampling}
Let $g \in L^2(G)$, let $0<A\le B < \infty$, and let $\LL$ be a closed subgroup of $G$. Suppose that $\LL$ is either discrete or co-compact. Then the following assertions are equivalent:
\begin{enumerate}[(i)]
\item $\set{E_{\gamma} T_{\lambda} g}_{\gamma \in \LL^\perp, \lambda \in \LL}$ is a frame for $L^2(G)$
  with bounds $A$ and $B$,
\item The subgroup $\LL$ is discrete and co-compact, hence a uniform lattice, and $\set{E_{\gamma} T_{\lambda} g}_{\gamma \in \LL^\perp, \lambda \in \LL}$ is a Riesz basis for $L^2(G)$
  with bounds $A$ and $B$.
\end{enumerate}
\end{theorem}

\begin{proof}
  The implication (ii)$\Rightarrow$(i) is trivial so we only have to consider
  (i)$\Rightarrow$(ii). 

  Assume first that the subgroup $\LL$ is discrete. Then $\LG=\LL^\perp$ is co-compact. We use the
  notation from the paragraphs preceding
  Theorem~\ref{thm:non-existence-cont-gabor-critical-sampling}. Then, as shown above, assertion (i)
  is equivalent to $\{\hat{\varphi}_x(\omega)\}$ being a frame for $\C$ for almost every $x \in X$,
  $\omega \in \widehat{\LL}$. However, a one element set is a frame if, and only if, it is a Riesz
  basis with the same bounds. Now, we repeat the argument above, but in the reverse direction using
  a Riesz sequence variant of Theorem~\ref{thm:fibers-characterization}. By \cite[Theorem
  4.3]{MR2578463} the scalar $\{\hat{\varphi}_x(\omega)\}$ is a Riesz basis for $\C$ if, and only
  if, the SI system $\seq{T_\lambda \varphi_x}_{\lambda \in \Lambda}$ is a Riesz basis in
  $\ell^2(\LL)$ with the same bounds. By a result in \cite{BowRos2014}, which generalizes
  \cite[Theorem 4.3]{MR2578463}, this is equivalent to $\{ T_{\gamma} \cF^{-1} T_{\lambda} g
  \}_{\gamma \in \LL^\perp, \lambda \in \LL}$ being a so-called \emph{continuous} Riesz
  basis. However, as shown in \cite{BowRos2014} \emph{continuous} Riesz sequences only exist if
  $\LL^\perp$ is discrete. Hence, $\{ T_{\gamma} \cF^{-1} T_{\lambda} g \}_{\gamma \in \LL^\perp,
    \lambda \in \LL}$ is actually a (discrete) Riesz basis. By unitarily equivalence, this implies
  that $\set{E_{\gamma} T_{\lambda} g}_{\gamma \in \LL^\perp, \lambda \in \LL}$ is a Riesz basis.

  Assume now that $\LL$ is co-compact. Then $\LG=\LL^\perp$ is discrete. Note
  that $\set{T_\lambda E_\gamma g}_{\gamma \in \LL^\perp, \lambda \in \LL}$ is
  unitarily equivalent to $\set{E_\gamma T_\lambda g}_{\gamma \in \LL^\perp,
    \lambda \in \LL}$ and repeat the argument above for the co-compact TI system
  $\set{T_\lambda E_\gamma g}_{\gamma \in \LL^\perp, \lambda \in \LL}$
\end{proof}

\begin{remark}
  In the extreme case $\LL=G$,
  Theorem~\ref{thm:non-existence-cont-gabor-critical-sampling} tell us
  that $\set{T_\lambda g}_{\lambda \in G}$ cannot be a frame for
  $L^2(G)$ \emph{unless} $G$ is discrete; if $G$ is discrete, then $\ghat$ is compact, and any $g
  \in L^2(G)$ with $0<A \le \abs{\hat{g}(\omega)}^2 \le B$ for a.e. $\omega \in
  \ghat$ will generate a frame $\set{T_\lambda g}_{\lambda \in G}$
  with bounds $A,B$. For discrete (irregular) Gabor systems in
  $L^2(\R^n)$ such questions are studied in \cite{MR1721808}. On the
  other hand, totality in $L^2(G)$ of the set $\set{T_\lambda g}_{\lambda \in G}$ is achievable for both discrete and non-discrete LCA groups $G$; \eg take any $g
  \in L^2(G)$ with $\hat{g}(\omega) \neq 0$ for a.e. $\omega \in
  \ghat$.
\end{remark}

Due to Theorem~\ref{thm:non-existence-cont-gabor-critical-sampling} we wish to
relax the ``critical'' density condition $\LL=\LG^\perp$, but in such a way that we still can apply Zak transform methods.
For regular Gabor systems 
\begin{equation}
\label{eq:lattice-Gabor-system-Rd}
\setprop{\myexp{2\pi i \gamma x }g(x-\lambda)}{\gamma \in
  \LG=A\Z^n,\lambda \in \LL=B\Z^n}
\end{equation}
in $L^2(\R^n)$ with $A,B \in \text{GL}_n(\R)$ rational density, where $A\Z^n \cap
B\Z^n$ is a full-rank lattice, is such a relaxation; for $n=1$ rational density
simply means $AB=\frac{p}{q} \in \Q$.
Our assumptions on the subgroups $\LL$ and $\LG$ in the remainder of
this section will mimic the setup of
rational density, and the characterization will depend on a vector-valued Zak
transform similar to the case of $L^2(\R^n)$ \cite{MR2294481,MR1460623,MR1448221}.

For a closed subgroup $H$ of $G$ the Zak transform $Z_H$ as introduced by Weil, albeit not under
this name, of a continuous function $f \in C_c(G)$ is:
\[ 
Z_H f(x,\omega) = \int_H
f(x+h)\overline{\omega(h)} \, dh \qquad \text{for a.e. } x \in X,\, \omega \in \widehat{H}.
\]
The Zak transform extends to a unitary operator from $L^2(G)$ onto $L^2(G/H \times \ghat/H^\perp)$
\cite{MR0165033,MR3019110}. We will use the Zak transform for discrete subgroups $H=\LG^\perp$,
where $\LG$ is co-compact, in which case, the convergence of the series $Z_H f(x,\alpha) =
\sum_{\alpha \in \LG^\perp} f(x+\alpha)\overline{\omega(\alpha)}$ is in the $L^2$-norm for a.e. $x$
and $\omega$. 

 The next result shows that the frame property of a Gabor system $\gaborset$ in $L^2(G)$ under certain assumptions of $\LL$ and $\LG$ is equivalent with the frame property of a family of associated Zak transformed variants of the Gabor system in $\C^p$. 

\begin{theorem}
  \label{thm:char-vector-valued-Zak-time}
 Let $g \in L^2(G)$, and let $0<A\le B < \infty$. Let $\Gamma$ be a closed, co-compact subgroup of
$\ghat$. Suppose that $\Lambda$ is a closed subgroup of $G$ 
such that $p:=\card{\Gamma^\perp / (\LL \cap \LG^\perp)}<\infty$. 
Let $\set{\chi_1,\dots, \chi_p}:=\anni(\widehat{\LG^\perp},\LL \cap
    \LG^\perp)$. Equip $\Lambda$ with some Haar measure $\mu_{\Lambda}$, and let $\mu_{\Lambda/(\Lambda\cap \Gamma^{\perp})}$ be the unique Haar measure on $\Lambda/(\Lambda\cap\Gamma^{\perp})$ such that for all $f\in L^1(\Lambda)$
\[ 
\int_{\Lambda} f(x) \, d\mu_{\Lambda}(x) = p \int_{\Lambda/(\Lambda\cap \Gamma^{\perp})} \sum_{\ell\in \Lambda\cap \Gamma^{\perp}} f(x+\ell) \, d\mu_{\Lambda/(\Lambda\cap \Gamma^{\perp})}(\dot{x}).
\] 
Also, we let $K\subset \Lambda$ denote a Borel section of $\Lambda\cap\Gamma^{\perp}$ in $\LL$ and $\mu_K$ be a measure on $K$ isometric to $\mu_{\Lambda/(\Lambda\cap \Gamma^{\perp})}$ in the sense of \eqref{eq:isom-borel}.
Then the following assertions are equivalent:
\begin{enumerate}[(i)]
\item $\set{E_{\gamma} T_{\lambda} g}_{\lambda \in
    \LL, \gamma \in \LG}$ is a frame for $L^2(G)$ with bounds $A$ and $B$,
\label{item:8}
\item 
$ A \, \Vert c \Vert_{\C^p}^2 \le \int_{K} \vert \langle c , \seq{Z_{\LG^\perp} g(x+\kappa,\omega\chi_i)}_{i=1}^p\rangle_{\C^p} \vert^2 \, d\mu_{K}(\kappa) \le B \, \Vert c \Vert_{\C^p}^2$ for all $c\in \C^p$, a.e. $x \in X$ and $\omega \in \widehat{\LG^\perp}$, where $X$ is a Borel section of $\LG^\perp$ in $G$,
\label{item:9}
\item  \[ 
  A \le \essinf_{(x,\omega) \in X \times \widehat{\LG^\perp}} \lambda_p(x,\omega), \quad
  B \ge \esssup_{(x,\omega) \in X \times \widehat{\LG^\perp}} \lambda_1(x,\omega),
  \] 
  where $\lambda_i(x,\omega)$ denotes the $i$-th largest eigenvalue value of the $p \times p$ matrix $\tilde{\cG}(x,\omega)$, whose $(i,j)$-th entry is
\[ \tilde{\cG}(x,\omega)_{(i,j)}=\int_{K} Z_{\LG^\perp} g(x+\kappa,\omega\chi_i) \overline{Z_{\LG^\perp} g(x+\kappa,\omega\chi_j)} d\mu_{K}(\kappa). \]
\end{enumerate}
\end{theorem}

\begin{proof}
  We first remark that $\anni(\widehat{\LG^\perp},\LL \cap \LG^\perp) \cong
  \Gamma^{\perp}/(\Lambda \cap \Gamma^{\perp})$ by Lemma
  \ref{le:finite-quotient-cong-annihilator}. This shows that $\set{\chi_1,
    \dots, \chi_p}$ is well-defined due to the assumption
  $p=\card{\Gamma^\perp/(\LL \cap \LG^\perp)}<\infty$.

  By Proposition~\ref{thm:char-frame-prop-LG-Gabor}, assertion (i) is equivalent
  to the sequence $\seq{\seq{g(x+\lambda+\alpha)}_{\alpha \in
      \LG^\perp}}_{\lambda \in \Lambda}$ being a frame for $\ell^2(\LG^\perp)$
  with bounds $A$ and $B$ for a.e. $x \in X$. Since $\Lambda \cap
  \Gamma^{\perp}$ is a subgroup of $\Lambda$, every $\lambda \in \LL$ can be
  written in a unique way as $\lambda = \mu + \kappa$ with $\mu \in \LL \cap
  \LG^\perp$ and $\kappa \in \Lambda/(\Lambda\cap \Gamma^{\perp})$.  Letting
  $\varphi_\kappa:=\seq{g(x+\alpha+\kappa)}_{\alpha \in \LG^\perp}$, we can write the
  above sequence as $\seq{T_\mu\varphi_\kappa}_{\mu \in \LL \cap \LG^\perp, \kappa \in
   \Lambda/(\Lambda\cap \Gamma^{\perp})}$. By assumption, this is a co-compact translation invariant system in
  $\ell^2(\Gamma^{\perp})$. The Fourier transform of $\varphi_\kappa \in
  \ell^2(\LG^\perp)$ is
  \[
  \hat{\varphi}_\kappa(\omega) = \sum_{\alpha \in \LG^\perp} g(x+\alpha+\kappa) \overline{\omega(\alpha)}
  \quad \text{for a.e. } \omega \in \widehat{\LG^\perp},
  \]
  hence $\hat{\varphi}_\kappa (\omega)=Z_{\LG^\perp} g(x+\kappa,\omega)$.  As above, we
  apply the fiberization techniques from Section~\ref{sec:fiberization} with
  $G=\LG^\perp$ and $\LH=\LL\cap \LG^\perp$. The relationship between the measures via Weil's formula in the assumption guarantees that the subgroups are equipped with the correct measures. Since the annihilator $\LH^\perp$
  in this case is $\anni(\widehat{\LG^\perp},\LL\cap\LG^\perp)$, the
  fiberization map~(\ref{eq:fiber-map}) is $\cT
  f(\omega)=\{\hat{f}(\omega\chi)\}_{\chi \in \anni(\widehat{\LG^\perp},\LL \cap
    \LG^\perp)}$ for $\omega \in \widehat{\LG^\perp}$. By
  Theorem~\ref{thm:fibers-characterization}, we see that assertion (i) is
  equivalent to the system
  \[
  \seq{\seq{Z_{\LG^\perp} g(x+\kappa,\omega\chi)}_{\chi \in
      \anni(\widehat{\LG^\perp},\LL \cap \LG^\perp)}}_{\kappa \in \Lambda/(\Lambda\cap \Gamma^{\perp})}
  \]
  being a frame in $\ell^2(\anni(\widehat{\LG^\perp},\LL \cap \LG^\perp))\cong \C^p$ with
  bounds $A$ and $B$ for a.e. $x \in X$ and $\omega \in \widehat{\LG^\perp}$. 
 This proves (i)$\Leftrightarrow$(ii). 

 The dual Gramian matrix $\tilde{\cG}(x,\omega)$ is a matrix representation of
 the frame operator of the system in (ii) which shows the equivalence
 (ii)$\Leftrightarrow$(iii).
\end{proof}

Under the assumption $p=\card{\Gamma^\perp / (\LL \cap \LG^\perp)}<\infty$, we can view
$\seq{Z_{\LG^\perp} g(x+\kappa ,\omega\chi)}_{\chi \in \anni(\widehat{\LG^\perp},\LL \cap
  \LG^\perp)}$ as a column vector in $\C^p$. This vector is sometimes called a vector-valued Zak
transform of $g$.  We remark that the quotient group $\Lambda/(\Lambda\cap \Gamma^{\perp})$ in
Theorem~\ref{thm:char-vector-valued-Zak-time} can be infinite, even uncountably infinite. If it is
finite, however, we have the following simplification.

\begin{corollary}
\label{thm:char-Zibulski-Zeevi}
  In addition to the assertions in Theorem \ref{thm:char-vector-valued-Zak-time}
  assume that $\Lambda$ is discrete, $q:= \card{\Lambda/(\Lambda\cap\Gamma^{\perp})}<\infty$ and let $\Lambda$ be equipped with the counting measure. Let $\kappa_i$,
  $i=1,\dots, q$, be a set of coset representatives of $\Lambda/(\Lambda\cap
  \Gamma^{\perp})$, and let
  $\set{\chi_1,\dots,\chi_p}:=\anni(\widehat{\LG^\perp},\LL \cap
  \LG^\perp)$. Then the following assertions are equivalent.
\begin{enumerate}[(i)]
\item $\set{E_{\gamma} T_{\lambda} g}_{\lambda \in
    \LL, \gamma \in \LG}$ is a frame for $L^2(G)$ with bounds $A$ and $B$,
\label{item:8b}
\item $\{\seq{Z_{\LG^\perp} g(x+\kappa_i,\omega\chi_j)}_{j=1}^p\}_{i=1}^q$ is a
  frame for $\C^p$ w.r.t.\ $p^{-1}$ times the counting measure, i.e., 
$ A \, \Vert c \Vert_{\C^p}^2 \le \tfrac{1}{p}\sum_{i=1}^{q} \vert \langle c , \seq{Z_{\LG^\perp} g(x+\kappa_i,\omega\chi_j)}_{j=1}^p\rangle_{\C^p} \vert^2 \, \le B \, \Vert c \Vert_{\C^p}^2$ for all $c\in \C^p$, \\
for a.e. $x \in X$ and $\omega \in \widehat{\LG^\perp}$, where $X$ is a
  Borel section of $\LG^\perp$ in $G$, \label{item:9b}
\item 
  \[ 
  A \le p^{-1} \essinf_{(x,\omega) \in X \times \widehat{\LG^\perp}} \sigma_p(x,\omega)^2, \quad
  B \ge p^{-1} \esssup_{(x,\omega) \in X \times \widehat{\LG^\perp}} \sigma_1(x,\omega)^2,
  \] 
  where $\sigma_k(x,\omega)$ denotes the $k$-th largest singular
  value of the $q \times p$ matrix  $\Phi(x,\omega)$, whose $(i,j)$-th entry is
  $Z_{\LG^\perp} g(x+\kappa_i,\omega \chi_j)$.
\label{item:10b}
\end{enumerate}
\end{corollary}

The matrix $p^{-1/2} \Phi(x,\omega)$ is called the Zibulski-Zeevi representation; it is the transpose of the matrix representation of the synthesis operator associated with the frame in Corollary~\ref{thm:char-Zibulski-Zeevi}(ii). This shows that the Zibulski-Zeevi representation is possible for Gabor systems with translation along a discrete  (but not necessarily co-compact) subgroup $\Lambda\subset G$ and modulation along a co-compact (but not necessarily discrete) subgroup $\Gamma\subset \ghat$. 

For lattice Gabor systems (\ref{eq:lattice-Gabor-system-Rd}) in $L^2(\R^n)$, Corollary~\ref{thm:char-Zibulski-Zeevi} reduces to \cite[Theorem 4.1]{MR2294481}. We remark that, in this case, the roles of $p$ and $q$ are the same as in \cite[Theorem 4.1]{MR2294481} which can be seen by an application of the second isomorphism theorem
\[ 
p= \absbig{\LG^\perp/(\Lambda\cap\Gamma^{\perp})}, \qquad
q= \absbig{\Lambda/(\Lambda\cap\Gamma^{\perp})}=\absbig{(\Lambda+\Gamma^{\perp})/\LG^\perp},
\]
and by noting that $\LG$ is assumed to be $\Z^n$ in \cite{MR2294481}.
 In particular, for regular Gabor systems in $L^2(\R)$ with time and
 frequency shift parameters $a$ and $b$, we have $ab = p/q \in \Q$,
 where $p$ and $q$ are relative prime. 

Using range functions, the equivalence of (i) and (ii) in all results in
this subsection can be formulated for basic
frames. For Corollary~\ref{thm:char-Zibulski-Zeevi} this simply reads: 
$\set{E_{\gamma} T_{\lambda} g}_{\lambda \in
    \LL, \gamma \in \LG}$ is a basic frame in $L^2(G)$ if, and only
  if, $\{\seq{Z_{\LG^\perp} g(x+\kappa_i,\omega\chi_j)}_{j=1}^p\}_{i=1}^q$ is a
  basic frame in $\C^p$.
In the following Example \ref{exa:prufer-group-zak} we apply this version of Corollary \ref{thm:char-Zibulski-Zeevi} to a non-discrete Gabor system and calculate its Zibulski-Zeevi representation. 

\begin{example}
\label{exa:prufer-group-zak}
Let $r \in \N$ be prime. We consider Gabor systems $\gaborset$ in
$L^2(\Z(r^\infty))$, where the Pr\"ufer $r$-group $G=\Z(r^{\infty})$, the discrete group of all $r^n$-roots of unity for all $n\in \N$, is equipped with the discrete topology and multiplication as group operation. Its dual group can be identified with the $r$-adic integers $\ghat=\mathbb{I}_r$. For $m,n \in \N$ define $\LL \subset \Z(r^\infty)$ and $\LG^\perp \subset \Z(r^\infty)$ as all $r^n$ and $r^m$ roots of unity, respectively. Then $\LL$ is a discrete, closed subgroup of $\Z(r^\infty)$, and $\LG$ is a co-compact, closed subgroup of $ \mathbb{I}_r$. Note that neither $\LL$ nor $\LG$ are uniform lattices. Let $X$ and $\Omega$ denote Borel sections of the subgroups $\LG^\perp \subset G$
and $\LL^\perp \subset \ghat$, respectively. For any $n,m \in \N$, we have
$m_G(X) m_{\ghat}(\Omega) = \infty$.
 Moreover,
\[ 
p= \absbig{\LG^\perp/(\Lambda\cap\Gamma^{\perp})}=r^{m-\min\set{m,n}}, 
\qquad
q= \absbig{\Lambda/(\Lambda\cap\Gamma^{\perp})}=r^{n-\min\set{m,n}}.
\]

If $m\ge n$, then $p=r^{m-n}$, $q=1$, and the Zibulski-Zeevi
representation is (up to scaling of $p^{-1/2}$) given as a (row) vector of length $p$: 
\[ \Phi(x,\omega)=\set{Z_{\Gamma^{\perp}}g(x,\omega\chi_j)}_{j=1}^p,\]
where $\{\chi_j\}_{j=1}^{p}=\mathsf{A}(\widehat{\Gamma^{\perp}},\Lambda)$.
On the other hand, if $n\ge m$, then $p=1$,  $q=r^{n-m}$,  and the Zibulski-Zeevi
representation $\Phi(x,\omega)=\set{Z_{\Gamma^{\perp}}(x+\kappa_i,\omega)}_{i=1}^q$ is
a (column) vector of length $q$, where $\{\kappa_i\}_{i=1}^{q}$ is a set of coset representatives of $\Lambda/\Gamma^{\perp}$.

Thus, for any $m,n\in \N$, the system $\gaborset$ is a frame for its closed linear span, i.e., a basic frame in $L^2(\Z(r^\infty))$, with bounds $A$ and $B$ if, and only if,
\[
A \le \frac{1}{p}\norm{\Phi(x,\omega)}^2 \le B 
\]
for almost every $x \in X$ and $\omega \in \widehat{\LG^\perp}$ for
which $\norm{\Phi(x,\omega)} \ne 0  $, where $\Phi(x,\omega)$ is given
as above.
\end{example}

\begin{remark}
  As an alternative to the Zak transform decomposition of $g$ used above in part (ii) of Theorem~\ref{thm:char-vector-valued-Zak-time} and Corollary~\ref{thm:char-Zibulski-Zeevi}, we can use a less time-frequency symmetric variant. The details are as follows. By a unitary transform on $\C^p$ the vector $\seq{1/\sqrt{p}Z_{\LG^\perp} g(x+\kappa,\omega\chi_i)}_{i=1}^p$ is mapped to the vector
  \begin{equation}
\psi_\kappa(x,\omega):=\seq{\sum_{\alpha \in \LL \cap \LG^\perp} g(x+\alpha+\kappa+\ell_i) \overline{\omega(\alpha)}}_{i=1}^p,\label{eq:zak-transform-alternative}
\end{equation}
where $\ell_i$, $i =1,\dots, p$, are distinct coset representatives of
$\LG^\perp/(\LL \cap \LG^\perp)$, and $\kappa \in K$. The assertions in
Theorem~\ref{thm:char-vector-valued-Zak-time} are, therefore, equivalent with
\[  A \, \Vert c \Vert_{\C^p}^2 \le \int_{K} \vert \langle c , \psi_{\kappa}(x,\omega)\rangle_{\C^p} \vert^2 \, d\mu_{K}(\kappa) \le B \, \Vert c \Vert_{\C^p}^2 \quad \text{for all} \ \ c\in \C^p,\]
a.e. $x \in X$ and $\omega \in \widehat{\LG^\perp}$, where $X$ is a
Borel section of $\LG^\perp$ in $G$. Here $\mu_K$ is the measure on
$K$ isometric to $\mu_{\Lambda/(\Lambda\cap\Gamma^{\perp})}$ (in the
sense of \eqref{eq:isom-borel}) such that for all $f\in L^1(\Lambda)$
\[ 
\int_{\Lambda} f(x) \, d\mu_{\Lambda}(x) = \int_{\Lambda/(\Lambda\cap \Gamma^{\perp})} \sum_{\ell\in \Lambda\cap \Gamma^{\perp}} f(x+\ell) \, d\mu_{\Lambda/(\Lambda\cap \Gamma^{\perp})}(\dot{x}); \] 
note that this is different from the measure $\mu_K$ used in Theorem \ref{thm:char-vector-valued-Zak-time}.
Then the assertions in Corollary~\ref{thm:char-Zibulski-Zeevi} are equivalent to the fact that
\[  A \, \Vert c \Vert_{\C^p}^2 \le \sum_{i=1}^{q} \vert \langle c , \psi_{\kappa_i}(x,\omega)\rangle_{\C^p} \vert^2  \le B \, \Vert c \Vert_{\C^p}^2 \quad \text{for all} \ \ c\in \C^p,\]
for a.e. $x \in X$ and $\omega \in \widehat{\LG^\perp}$, where $X$ is
a Borel section of $\LG^\perp$ in $G$. 
\end{remark}


If we switch the assumptions on $\LL$ and $\LG$ and consider TI
systems of the form $\set{T_\lambda E_\gamma g}_{\gamma \in \LG,
  \lambda \in \LL}$, we obtain the following variant of Proposition~\ref{thm:char-frame-prop-LG-Gabor}. 
\begin{proposition}
\label{thm:char-frame-prop-LL-Gabor}
Let $g \in L^2(G)$, and let $0<A\le B < \infty$. Let $\LL$ be a closed, co-compact subgroup of
$G$, and let
$(\LG,\Sigma_\LG,\mu_\LG)$ be an admissible measure space in $\ghat$. The following assertions are
equivalent:
\begin{enumerate}[(i)]
\item $\set{E_{\gamma} T_{\lambda} g}_{\gamma \in \LG, \lambda \in \LL}$ is a frame for $L^2(G)$
  with bounds $A$ and $B$, 
\item $\seq{\seq{\hat{g}(\omega\gamma\beta)}_{\beta \in \LL^\perp}}_{\gamma \in \LG}$ is a frame for
  $\ell^2(\LL^\perp)$ with bounds $A$ and $B$ for a.e. $\omega \in \Omega$, where $\Omega$ is a Borel
  section of $\LL^\perp$ in $G$.
\end{enumerate}
\end{proposition}

From Proposition~\ref{thm:char-frame-prop-LL-Gabor} we get the following variant of Theorem~\ref{thm:char-vector-valued-Zak-time}; we leave the corresponding formulation of Corollary~\ref{thm:char-Zibulski-Zeevi} to the reader.

\begin{theorem}
  \label{thm:char-vector-valued-Zak-frequency}
 Let $g \in L^2(G)$, and let $0<A\le B < \infty$. Let $\Lambda$ be a closed, co-compact subgroup of
$G$. Suppose that $\Gamma$ is a closed subgroup of $\ghat$ 
such that $p:=\card{\Lambda^\perp / (\Gamma \cap \Lambda^\perp)}<\infty$. 
Let $\set{\chi_1,\dots, \chi_p}:=\anni(\widehat{\Lambda^\perp},\Gamma \cap
    \Lambda^\perp)$. Equip $\Gamma$ with some Haar measure $\mu_{\Gamma}$, and let $\mu_{\Gamma/(\Gamma\cap \Lambda^{\perp})}$ be the unique Haar measure over $\Gamma/(\Gamma\cap\Lambda^{\perp})$ such that for all $f\in L^1(\Gamma)$
\[ 
\int_{\Gamma} f(x) \, d\mu_{\Gamma}(x) = p \int_{\Gamma/(\Gamma\cap \Lambda^{\perp})} \sum_{\ell\in \Gamma\cap \Lambda^{\perp}} f(x+\ell) \, d\mu_{\Gamma/(\Gamma\cap \Lambda^{\perp})}(\dot{x}).
\] 
Also, we let $K\subset \Gamma$ denote a Borel section of $\Gamma\cap\Lambda^{\perp}$ in $\Gamma$ and $\mu_K$ be a measure on $K$ isometric to $\mu_{\Gamma/(\Gamma\cap \Lambda^{\perp})}$ in the sense of \eqref{eq:isom-borel}.
Then the following assertions are equivalent:
\begin{enumerate}[(i)]
\item  $\set{E_{\gamma} T_{\lambda} g}_{\gamma \in \LG, \lambda \in \LL}$ is a frame for $L^2(G)$
  with bounds $A$ and $B$, 
\label{item:8c}
\item 
$ A \, \Vert c \Vert_{\C^p}^2 \le \int_{K} \vert \langle c , \seq{Z_{\Lambda^\perp} \hat g(\omega\kappa,x+\chi_i)}_{i=1}^p\rangle_{\C^p} \vert^2 \, d\mu_{K}(\kappa) \le B \, \Vert c \Vert_{\C^p}^2$ for all $c\in \C^p$,
a.e. $\omega \in \Omega$ and $x \in \widehat{\Lambda^\perp}$, where $\Omega$ is a Borel section of $\Lambda^\perp$ in $\ghat$.  
\end{enumerate}
\end{theorem}


\subsection{Characterizations of dual Gabor frames} 
\label{sec:char-equat-dual}

 By a result on so-called characterizing equations from~\cite{JakobsenReproducing2014}, we now
characterize when two semi co-compact Gabor systems are dual frames. Using the equivalence of frame properties for systems $\gaborset$ and $\{ T_{\gamma} \cF^{-1} T_{\lambda} g \}_{\gamma \in \LG,
  \lambda \in \LL}$ with generator $g\in L^2(G)$ yields the following characterizing equations in the time domain.
\begin{theorem}[$\!\!$\cite{JakobsenReproducing2014}] 
\label{th:dual-Gabor-time} 
Let $\Gamma$ be a closed, co-compact subgroup of $\ghat$, and let $(\Lambda,\Sigma_\Lambda,\mu_\Lambda)$ be an admissible measure space in $G$.
Suppose that the two systems $\{E_{\gamma}T_{\lambda}g\}_{\gamma\in\Gamma, \lambda\in\Lambda}$ and
$\{E_{\gamma}T_{\lambda}h\}_{\gamma\in\Gamma, \lambda\in\Lambda}$ are Bessel systems. Then the
following statements are equivalent:
\begin{enumerate}[(i)]
\item $\{E_{\gamma}T_{\lambda}g\}_{\gamma\in\Gamma, \lambda\in\Lambda}$ and
  $\{E_{\gamma}T_{\lambda}h\}_{\gamma\in\Gamma, \lambda\in\Lambda}$ are dual frames for $L^2(G)$,
 \item for each $\alpha \in \LG^{\perp}$ we have
\begin{equation}
s_\alpha(x):= \int_{\Lambda} \overline{g(x-\lambda-\alpha)} h(x-\lambda) \, d\mu_{\Lambda}(\lambda)
= \delta_{\alpha,0} \quad \almoste\ x \in G,\label{eq:def-s-alpha}
\end{equation}
\end{enumerate}
\end{theorem}
If we want to stress the dependence of the generators $g$ and $h$ in (\ref{eq:def-s-alpha}), we use
the notation $s_{g,h,\alpha}:G \to \C$.

\begin{corollary}
\label{th:dual-Gabor-time-tight} 
Let $\Gamma$ be a closed, co-compact subgroup of $\ghat$, and let $(\Lambda,\Sigma_\Lambda,\mu_\Lambda)$ be an admissible measure space in $G$.  The
family $\{E_{\gamma}T_{\lambda}g\}_{\gamma\in\Gamma, \lambda\in\Lambda}$ is an $A$-tight frame for
$L^2(G)$ if and only if $s_{g,g,\alpha}(x)= A\,\delta_{\alpha,0}$ a.e. for each $\alpha \in
\LG^{\perp}$.
\end{corollary}

\begin{example}
\label{exa:parseval-LG-being-G}
  Let $g \in L^2(G)$ and consider $\{E_{\gamma}T_{\lambda}g\}_{\gamma\in\ghat,\lambda\in \LL}$, where $(\Lambda,\Sigma_\Lambda,\mu_\Lambda)$ be an admissible measure space in $G$. By Corollary~\ref{th:dual-Gabor-time-tight} we see that  
$\{E_{\gamma}T_{\lambda}g\}_{\gamma\in\ghat,\lambda\in \LL}$ is a Parseval frame for $L^2(G)$ if, and only if, for a.e. $x\in G$
\begin{equation}
\int_\LL \abs{g(x-\lambda)}^2 \, d\mu_\LL(\lambda) =  1.
\label{eq:LG-being-G-parseval}
\end{equation}
If we take $\LL=G$ with the Haar measure, then equation~(\ref{eq:LG-being-G-parseval}) becomes simply $\norm{g}=1$ which is the well-known inversion formula for the short-time Fourier transform \cite{MR1601095,MR1843717}.

Suppose now that $G$ contains a uniform lattice. Take $\LL$ as a uniform lattice in $G$, and let $X$ denote a (relatively compact) Borel section of $\LL$ in $G$. 
Equation~(\ref{eq:LG-being-G-parseval}) becomes
\[ 
 \sum_{\lambda \in \LL} \abs{g(x-\lambda)}^2 =  \card{X}^{-1}.
\] 
Let $g_1, \dots , g_r \in L^2(G)$ be functions positive on $X$ with support $\supp g_i \subset \overline{X}$ so that $g_i$ is constant on $X$ for at least one index $i$. 
Following \cite{OleSaySon2014}, the function on $G$ defined by the $r$-fold convolution
\[ 
W_r := g_1 \mathds{1}_X \ast g_2 \mathds{1}_X \ast \ldots \ast g_r \mathds{1}_X 
\] 
is called a weighted B-spline of order $r$. As shown in \cite{OleSaySon2014}, the function $W_r$ is non-negative and satisfies a partition of unity condition up to a constant, say $\sum_{\lambda \in \LL} W_r(x-\lambda) = C_r$. Take $g \in L^2(G)$ so that \[
\abs{g(x)}^2=\frac{1}{C_r \card{X}} W_r(x) , \quad \text{\eg} \quad g(x)=\frac{1}{(C_r \card{X})^{1/2}} \sqrt{W_r(x)} .
\]
 Then  $\{E_{\gamma}T_{\lambda}g\}_{\gamma\in\ghat,\lambda\in \LL}$ is a Parseval frame.
\end{example}

Viewing Gabor systems as unitarily equivalent to $\set{T_{\lambda} E_{\gamma}
  g}_{\gamma\in\LG,\lambda\in \LL}$, we arrive at characterizing equations for duality in the
frequency domain.

\begin{theorem}[$\!\!$\cite{JakobsenReproducing2014}] 
\label{th:dual-Gabor} 
Let $\Lambda$ be a closed, co-compact subgroup of $G$, and let $(\Gamma,\Sigma_{\Gamma},\mu_{\Gamma})$ be an admissible measure space in $\ghat$. Suppose that
the two systems $\{E_{\gamma}T_{\lambda}g\}_{\gamma\in\Gamma, \lambda\in\Lambda}$ and
$\{E_{\gamma}T_{\lambda}h\}_{\gamma\in\Gamma, \lambda\in\Lambda}$ are Bessel systems. Then the
following statements are equivalent:
\begin{enumerate}[(i)]
\item $\{E_{\gamma}T_{\lambda}g\}_{\gamma\in\Gamma, \lambda\in\Lambda}$ and
  $\{E_{\gamma}T_{\lambda}h\}_{\gamma\in\Gamma, \lambda\in\Lambda}$ are dual frames for $L^2(G)$,
 \item for each $\beta \in \Lambda^{\perp}$ we have
\begin{equation}
t_\beta(\omega):= \int_{\Gamma} \overline{\hat{g}(\omega\gamma^{-1}\beta^{-1})} \hat{h}(\omega\gamma^{-1}) \,
d\mu_{\Gamma}(\gamma) = \delta_{\beta,1} \quad \almoste\ \omega \in \ghat.
\label{eq:def-t-alpha}
\end{equation}
\end{enumerate}
\end{theorem}
As for $s_{g,h,\alpha}$ we write $t_{g,h,\beta}:\ghat \to \C$ for $t_\beta$ in
\eqref{eq:def-t-alpha} if we want to stress the dependence of the generators $g$ and $h$.

\begin{corollary}
\label{th:dual-Gabor-freq-tight} 
Let $\Lambda$ be a closed, co-compact subgroup of $G$, and let $(\Gamma,\Sigma_{\Gamma},\mu_{\Gamma})$ be an admissible measure space in $\ghat$.  The family
$\{E_{\gamma}T_{\lambda}g\}_{\gamma\in\Gamma, \lambda\in\Lambda}$ is an $A$-tight frame for $L^2(G)$
if and only if $t_{g,g,\beta}(x)= A\,\delta_{\beta,1}$ a.e. for each $\beta \in \LL^{\perp}$.
\end{corollary}

Let us now consider co-compact Gabor systems, i.e., we take both $\Lambda$ and $\Gamma$ to be closed, co-compact
subgroups. We first remark that in this case, under the Bessel system assumption, we have
equivalence of conditions (\ref{eq:def-s-alpha}) and (\ref{eq:def-t-alpha}). More importantly,
$s_{g,h,\alpha}$ and $t_{g,h,\beta}$ can be written as a Fourier series. 
\begin{remark} \label{rem:t-alpha-unif-bounded-and-fourier-series} 
  \begin{enumerate}[(i)]
\item     For $g,h\in L^2(G)$ assume that two co-compact Gabor systems
    $\gaborset$ and $\gaborset[h]$ are Bessel systems with bounds
    $B_g$ and $B_h$, respectively. By an application of
    Cauchy-Schwarz' inequality and
    \cite[Proposition~3.3]{JakobsenReproducing2014}, we see that
    $s_{g,h,\alpha} \in L^\infty(G)$; to be precise:
    \[
    \abs{s_{g,h,\alpha}(x)} \le B_g^{1/2} B_h^{1/2} \quad \text{for
      a.e. } x \in G.
    \]
  \item Note that $s_{g,h,\alpha}:G\to \C$ is
    $\Lambda$-periodic. Furthermore, $G/\Lambda$ is compact and
    $s_{\alpha}$ is uniformly bounded, we can therefore consider
    $s_{g,h,\alpha}$ as a function in $L^2(G/\Lambda)$ and its Fourier
    series is given by
    \[
    s_{g,h,\alpha}(x)= \sum_{\beta \in \LL^\perp} c_{\alpha,\beta}
    \beta(x) \quad \text{with } c_{\alpha,\beta} = \int_{G/\Lambda}
    s_{g,h,\alpha} (\dot{x}) \overline{\beta(\dot{x})} \,
    d\dot{x}. 
    \]
    We can compute the Fourier coefficients $c_{\alpha,\beta}$
    directly using Weil's formula:
    \begin{align}
      c_{\alpha,\beta} & = \int_{G/\Lambda} s_{g,h,\alpha} (\dot x) \overline{\beta(\dot x)} \, d\dot x = \int_{G/\Lambda} \int_{\Lambda} \overline{g(x-\lambda-\alpha)} \, h(x-\lambda) \, \overline{\beta(x-\lambda)} \, d\lambda  \, d\dot x \nonumber \\
      & = \int_{G} h(x) \, \overline{\beta(x) g(x-\alpha)} \, dx =
      \langle h, E_{\beta} T_{\alpha}
      g\rangle. \label{eq:s-Fourier-coeff}
    \end{align}
  \item Similarly, we find $t_{g,h,\beta}(\omega) =
    \sum_{\alpha\in\Gamma^{\perp}} \langle \hat h, E_{\alpha}T_{\beta}
    \hat g \rangle \, \omega(\alpha)$.
  \end{enumerate}
\end{remark}

\section{The frame operator of Gabor systems}
\label{sec:Gabor-frame-operator}

Let us begin with the definition of the frame operator. Let $g\in L^2(G)$, and let $\Lambda\subset G$, $\Gamma\subset \ghat$ be closed subgroups. If $\gaborset$ is a Bessel system, the frame operator introduced in (\ref{eq:def-frame-op-in-H}) reads:
\[
S \equiv S_{g,g} : L^2(G)\to L^2(G), \quad S = \int_{\Gamma} \int_{\Lambda} \langle \, \cdot \, , \gabor
\rangle \gabor \, d\lambda \, d\gamma,
 \] 
given weakly by
\begin{align*}
\langle S f_1,f_2 \rangle & = \int_{\Gamma} \int_{\Lambda} \langle f_1, \gabor\rangle\langle
\gabor,f_2\rangle \, d\lambda \, d\gamma \quad \forall f_1,f_2\in L^2(G).
\end{align*}
Similarly, for two Gabor Bessel systems generated by the functions $g,h\in L^2(G)$, we introduce the
operator
\begin{equation}
S_{g,h} : L^2(G)\to L^2(G), \quad S_{g,h} = \int_{\Gamma} \int_{\Lambda} \langle \, \cdot \, , \gabor \rangle \gabor[h] \,  d\lambda \, d\gamma.
\label{eq:def-Gabor-frame-op-2-gen}
\end{equation}
We follow the Gabor theory tradition, referring to this operator as a (mixed) frame operator.  If we want to
emphasize the role of $\Lambda$ and $\Gamma$, we denote this operator $S_{g,h,\Lambda,\Gamma}$, where $\Lambda$ specifies the translation subgroup and $\Gamma$ the modulation subgroup. 

As in Gabor theory on $L^2(\R^n)$, it is straightforward to show that the frame operator commutes with time-frequency shifts with respect to the groups $\Lambda$ and $\Gamma$. 
\begin{lemma}
\label{thm:frame-op-commutes-time-freq}
Suppose that $\Gamma$ and $\Lambda$ are closed subgroups. 
Let $g,h\in L^2(G)$ and let $\gaborset$, $\gaborset[h]$ be Bessel
systems. Then, for all $\gamma \in \Gamma$ and $\lambda \in \Lambda$, the following holds:
\begin{enumerate}[(i)]
\item $S_{g,h} E_{\gamma}T_{\lambda} = E_{\gamma}T_{\lambda} S_{g,h}$,
\item If $\gaborset$ is a frame, then
\[ S^{-1} E_{\gamma}T_{\lambda} = E_{\gamma}T_{\lambda} S^{-1}. \]\end{enumerate}
\end{lemma} 

Lemma~\ref{thm:frame-op-commutes-time-freq} implies that the canonical
dual of a Gabor frame again is a Gabor system of the form $\gaborset[h]$, where $h=S^{-1}g$. Finally, we note that by a direct application of the Plancherel theorem, one can show that for all $f_1,f_2\in L^2(G)$,
  \[
  \langle  S_{g,h,\Lambda,\Gamma} f_1,f_2\rangle = \langle S_{\hat g,\hat h,\Gamma,\Lambda}
  \hat{f}_1,\hat{f}_2\rangle ,
  \]
where $\LL$ and $\LG$ are only assumed to be measure spaces.

\subsection{Feichtinger's algebra}
\label{sec:feichtingers-algebra}

In applications of our results, one often needs to show that the Gabor
system $\gaborset$ generated by $g\in L^2(G)$ constitutes a Bessel
family. This task, however, can be non-trivial, and even if $g$
generates a Bessel system for subgroups $\Lambda_1$ and $\Gamma_1$, it
may not generate a Bessel system for another pair of translation and
modulation groups $\Lambda_2$ and $\Gamma_2$. A solution to this
problem is to consider functions in the \emph{Feichtinger algebra}
$S_0(G)$.  It follows from \cite[Theorem 3.3.1]{MR1601091} that Gabor
systems $\gaborset$ with respect to any two uniform lattices $\Lambda$ and
$\Gamma$ in $\R^n$ generated by functions in $S_0(\R^n)$ are Bessel
systems. The proof relies on properties of the Wiener-Amalgam spaces. The purpose of this section is to give an alternate proof in the setting of LCA groups that any $g\in
S_0(G)$ generates a Bessel system $\gaborset$ for any two closed subgroups
$\Lambda\subset G$ and $\Gamma\subset \ghat$. 

Let $g\in C_c(G)$ be a non-zero function with $\mathcal F g \in L^1(\ghat)$. The Feichtinger
algebra $S_0(G)$ is then defined as follows:
\[ S_0(G) := \big\{ f: G\to \C \, : \, f\in L^1(G) \ \text{and} \
\int_{G} \int_{\ghat} \vert \mathcal V_g f(x,\omega) \vert \, d\omega
\, dx < \infty \big\},
\]
where $\mathcal V_g f(x,\omega) := \int_G f(t) \overline{\omega(t) g(t-x)} \, dt$ is the short time
Fourier transform of $f$ with the window $g$. Equipped with the norm
$\Vert f \Vert_{S_0} := \int_{G\times \ghat} \vert \mathcal V_g
f(x,\omega) \vert  d\omega dx$, the function space $S_0(G)$ is a Fourier-invariant Banach space that is dense in
$L^2(G)$ and whose members are continuous and integrable functions.
Moreover, $S_{0}(G)$ is continuously embedded in $L^{1}(G)$, that is, there
exists a constant $C>0$ such that
\[ 
\Vert f \Vert_{L^{1}(G)} \le C \Vert f \Vert_{S_{0}(G)} \quad \text{for all $f\in S_0(G)$.} 
\] 
If $g,h\in S_0(G)$, then $\mathcal V_g h \in S_0(G\times \ghat)$. Furthermore, for any closed subgroup $H\subset G$ the \emph{restriction
  mapping}
\[ 
\mathcal R_H: S_0(G) \to S_0(H), \ (\mathcal R_H f)(x) := f(x), \ \ x\in H
\]
is a surjective, bounded and linear operator. We refer the reader to
\cite{MR1601107,MR643206,MR1601091} for a detailed introduction to $S_0(G)$. 

In order to prove Theorem \ref{thm:feichtinger-implies-bessel}, we need the following two results. Lemma \ref{le:bounded-subgroup-integral-over-S0-function} relies on properties (ii) and (iv) from above, whereas Lemma \ref{le:lemma22-for-S0-gabor} is an adaptation of \cite[Lemma
2.2]{JakobsenReproducing2014}.

\begin{lemma} \label{le:bounded-subgroup-integral-over-S0-function} Let $H$ be a closed subgroup in $G$ and let $a\in G$, $g\in S_0(G)$. Then there exists some constant $K_H>0$ which depends on $H$ such that
\[ \int_{H} \vert g(x-a) \vert \, d\mu_{H}(x) \le K_H \, \Vert g \Vert_{S_0(G)} \ \ \text{for all } a\in G.\]
\end{lemma}
\begin{proof} The result follows from the fact that $S_0(H)$ is continuously embedded in $L^1(H)$ and the boundedness of the restriction mapping: 
\begin{align*}
 \int_{H} \vert g(x-a) \vert \, d\mu_{H}(x) & = \Vert \mathcal R_H (T_a g) \Vert_{L^1(H)} \le C \, \Vert \mathcal R_H (T_a g) \Vert_{S_0(H)} \\ & \le C C_H \, \Vert T_a g \Vert_{S_0(G)} = C C_H \, \Vert g \Vert_{S_0(G)}.
\end{align*}
Here we also used that the $S_0$-norm is invariant under translation. Now take $K_H = CC_H$.
\end{proof}

\begin{lemma} \label{le:lemma22-for-S0-gabor}
Let $g\in L^2(G)$ and $\Gamma\subset \ghat$ be a closed subgroup. For all $f\in C_c(G)$
\begin{equation} \label{eq:lemma22-for-S0-gabor} \int_{\Gamma} \vert \langle f, E_{\gamma} T_{\lambda} g\rangle \vert^2 \, d\mu_{\Gamma}(\gamma) = \int_G \int_{\Gamma^{\perp}} f(x) \overline{f(x-\alpha)} \, \overline{T_{\lambda}g(x)} \, T_{\lambda}g(x-\alpha)\, d\mu_{\Gamma^{\perp}}(\alpha) \, d\mu_G(x).
\end{equation} 
\end{lemma}

With these results in hand, we can prove that functions in $S_0(G)$ always generate Gabor Bessel systems.

\begin{theorem}
\label{thm:feichtinger-implies-bessel}
 Let $g\in S_0(G)$ and let $\Lambda\subset G$ and $\Gamma\subset \ghat$ be closed subgroups. Then $\gaborset$ is a Bessel system with bound $B=K_{\Lambda,\Gamma} \, \Vert g \Vert_{S_0(G)}^2$, where $K_{\Lambda,\Gamma}$ is a constant that only depends on $\Lambda$ and $\Gamma$.
\end{theorem}
\begin{proof}
From Lemma \ref{le:lemma22-for-S0-gabor} follows that for all $f\in C_c(G)$:
\begin{align*}
& \int_{\Lambda}\int_{\Gamma} \vert \langle f, E_{\gamma}T_{\lambda}g\rangle \vert^2 \, d\gamma \, d\lambda \\
 & = \int_{\Lambda} \int_G \int_{\Gamma^{\perp}} f(x) \overline{f(x-\alpha)} \, \overline{T_{\lambda}g(x)} \, T_{\lambda}g(x-\alpha)\, d\alpha \, dx \, d\lambda \\
  & = \int_{\Lambda} \int_{G/\Gamma^{\perp}} \int_{\Gamma^{\perp}} \int_{\Gamma^{\perp}} g(x-\lambda-\alpha) \, \overline{f(x-\alpha)}\, \overline{g(x-\lambda-\alpha')} \, f(x-\alpha') \, d\alpha \, d\alpha' \, d\dot x \, d\lambda.
\end{align*}
In the latter equality we used Weil's formula and a change of variables $\alpha+\alpha' \mapsto \alpha$.
An application of the triangle inequality and the Cauchy-Schwarz inequality now yields the following estimate:
\begin{align}
& \int_{\Lambda}\int_{\Gamma} \vert \langle f, E_{\gamma}T_{\lambda}g\rangle \vert^2 \, d\mu_{\Gamma}(\gamma) \, d\mu_{\Lambda}(\lambda) \nonumber \\
& \le  \int_{G/\Gamma^{\perp}} \int_{\Gamma^{\perp}}\int_{\Gamma^{\perp}} \big\vert f(x-\alpha) f(x-\alpha') \big\vert \int_{\Lambda} \big\vert g(x-\lambda-\alpha) g(x-\lambda-\alpha') \big\vert \, d\lambda \, d\alpha \, d\alpha' d\dot x \nonumber \\
& \le  \int_{G/\Gamma^{\perp}} \Big( \int_{\Gamma^{\perp}} \big\vert  f(x-\alpha) \big\vert^2  \int_{\Lambda} \int_{\Gamma^{\perp}}\big\vert g(x-\lambda-\alpha) g(x-\lambda-\alpha') \big\vert \, d\alpha' \, d\lambda \,  d\alpha\Big)^{1/2} \nonumber \\
& \qquad \qquad  \Big( \int_{\Gamma^{\perp}} \big\vert   f(x-\alpha') \big\vert^2  \int_{\Lambda} \int_{\Gamma^{\perp}}\big\vert g(x-\lambda-\alpha) g(x-\lambda-\alpha') \big\vert \, d\alpha \, d\lambda \,  d\alpha' \Big)^{1/2} d\dot x. \label{eq:S0-always-bessel-estimate-eq1}
\end{align}
The order of integration can be rearranged due to Tonelli's
theorem. We now apply Proposition~\ref{le:bounded-subgroup-integral-over-S0-function} 
to the two innermost integrals and find that there exists a constant $K_{\Lambda,\Gamma}>0$ such that
\[ \int_{\Lambda}  \int_{\Gamma^{\perp}} \big\vert g(x-\lambda-\alpha) g(x-\lambda-\alpha') \big\vert \, d\alpha \, d\lambda = \int_{\Lambda} \vert g(x-\lambda-\alpha') \big\vert \int_{\Gamma^{\perp}} \big\vert g(x-\lambda-\alpha) \vert \, d\alpha \, d\lambda \le K_{\Lambda,\Gamma} \, \Vert g \Vert_{S_0(G)}^2, \]
where $\alpha'\in \Gamma^{\perp}$.
 Using this inequality in \eqref{eq:S0-always-bessel-estimate-eq1} yields the Bessel bound:
\begin{align*}
 \int_{\Lambda}\int_{\Gamma} \vert \langle f, E_{\gamma}T_{\lambda}g\rangle \vert^2 & \, d\mu_{\Gamma}(\gamma) \, d\mu_{\Lambda}(\lambda) \\
& \le  \int_{G/\Gamma^{\perp}} \Big( \int_{\Gamma^{\perp}} \big\vert   f(x-\alpha) \big\vert^2 K_{\Lambda,\Gamma} \, \Vert g \Vert_{S_0(G)}^2 \Big)^{1/2} \nonumber \\
& \qquad \qquad  \Big( \int_{\Gamma^{\perp}} \big\vert f(x-\alpha') \big\vert^2 K_{\Lambda,\Gamma} \, \Vert g \Vert_{S_0(G)}^2  \Big)^{1/2} d\mu_{G/\Gamma^{\perp}}(\dot x) \\
& =  \ K_{\Lambda,\Gamma} \, \Vert g \Vert_{S_0(G)}^2 \, \Vert f \Vert_{L^2(G)}^2.\end{align*}
Since $C_c(G)$ is dense in $L^2(G)$, the result follows.
\end{proof}

\subsection{The Walnut representation of the frame operator}
\label{sec:walnut-repr-frame-oper}

The continuous Gabor frame operator associated with \emph{semi} co-compact
Gabor systems defined in (\ref{eq:def-Gabor-frame-op-2-gen}) can be
converted into a \emph{discrete} transform called the Walnut
representation. The Walnut representation plays an important role the
usual discrete (lattice) theory of Gabor analysis.  For Gabor theory
on $L^2(\mathbb R)$ the result goes back to \cite{MR1155734} and is
also presented in \cite{MR1843717}. See \cite{MR1814424} for a
detailed analysis of the convergence properties of the Walnut
representation in $L^2(\R)$.

In order to state our version of the Walnut representation, we need to
introduce two dense subspaces of $L^2(G)$:
\begin{align}
  \label{eq:def-D-time}
  \cD_s := \big\{ f \in L^2 (G) \, : \, {f} \in L^\infty(G) \text{ and }
    \supp {f} \text{ is compact in } G \big\} \\
\intertext{and} 
 \label{eq:def-D-freq}
  \cD_t := \big\{ f \in L^2 (G) \, : \, \hat{f} \in L^\infty(\ghat) \text{ and }
    \supp \hat{f} \text{ is compact in } \ghat \big\}.
\end{align}
Recall also the definition of $s_{\alpha}$ and $t_{\beta}$ from
\eqref{eq:def-s-alpha} and \eqref{eq:def-t-alpha}, respectively.
\begin{theorem} 
  \label{th:Gabor-walnut-dense} 
  Let $g,h \in L^2(G)$.  Let $\Gamma$ be a closed, co-compact subgroup of
  $\ghat$, and let $(\Lambda,\Sigma_\Lambda,\mu_\Lambda)$ be an admissible
  measure space in $G$. Suppose that \gaborset and \gaborset[h] are Bessel systems,
  and let $S_{g,h}$ be the associated mixed frame operator. Then
\begin{equation}
\label{eq:gabor-frame-walnut}
 S_{g,h}f = \sum_{\alpha\in \Gamma^{\perp}} \!
    M_{s_{\alpha}} T_{\alpha} f \qquad \text{for all $f\in \cD_s$}, 
\end{equation} 
with unconditional, norm convergence in $L^2(G)$.
\end{theorem}
\begin{proof}
By the proof of the main result in \cite{JakobsenReproducing2014}, we have that for all $f_1,f_2\in \cD_s$,
\begin{align*}
\langle S_{g,h} f_1,f_2 \rangle & = \int_{\Lambda} \int_{\Gamma} \langle f_1,E_{\gamma} T_{\lambda} g\rangle \langle E_{\gamma} T_{\lambda} h, f_2 \rangle \, d\gamma \, d\lambda \\
& =  \sum_{\alpha\in \Gamma^{\perp}} \big\langle M_{s_{\alpha}} T_{\alpha} f_1,f_2\big\rangle.
\end{align*} 
Moreover, the convergence is absolute and thus unconditionally. Because
$\cD_s$ is dense in $L^2(G)$ spaces we have that $\langle S_{g,h} f_1,f_2 \rangle= \sum_{\alpha\in \Gamma^{\perp}} \big\langle M_{s_{\alpha}} T_{\alpha} f_1,f_2\big\rangle$ holds for all $f_2\in
L^2(G)$. By the Orlicz-Pettis Theorem (see, e.g., \cite{MR737004}), this implies unconditional $L^2$-norm
convergence for \eqref{eq:gabor-frame-walnut}.
\end{proof}

\begin{remark}
  If we assume $g,h \in S_0(G)$, then \eqref{eq:gabor-frame-walnut} extends to all of $L^2(G)$.
\end{remark}

\begin{remark}
\label{rem:Gabor-walnbut-dense-freq}
In Theorem~\ref{th:Gabor-walnut-dense}, if we instead assume that $\LL$ is a closed, co-compact
subgroup of $G$ and that $(\LG,\Sigma_\LG,\mu_\LG)$ is an admissible measure space in $\ghat$, then 
\begin{equation}
\label{eq:gabor-frame-walnut-freq}
 \cF S_{g,h}f = \sum_{\beta\in \LL^{\perp}} \!
    M_{t_{\beta}} T_{\beta} \hat{f} \qquad \text{for all $f\in \cD_t$} 
\end{equation} 
holds.
\end{remark}

We can now easily show the following result.
\begin{corollary} \label{cor:frame-implies-bounded-calderon} 
\begin{enumerate}[(i)]
\item Under the assumptions of Theorem~\ref{th:Gabor-walnut-dense} and if $\gaborset$
  is a frame with bounds $A$ and $B$, then
  \[
  A \le \int_{\Lambda} \vert g(x+\lambda) \vert^2 \, d\lambda \le B \qquad a.e. \ x \in G.
  \]
\item Under the assumptions of Remark~\ref{rem:Gabor-walnbut-dense-freq} and if $\gaborset$
  is a frame with bounds $A$ and $B$, then
  \[ 
  A \le \int_{\Gamma} \vert \hat g(\omega\gamma) \vert^2 \, d\gamma \le B \qquad  a.e. \ \omega \in
  \ghat.
  \]
\end{enumerate}
\noindent In either case, if $\gaborset$ is a Bessel system with bound $B$, then the upper bound holds.
\end{corollary}
\begin{proof}
  If $\gaborset$ is a frame, then, in particular,
  \[
  A \, \Vert f \Vert^2 \le \langle S_{g,g} f, f\rangle \le B \, \Vert f \Vert^2 \qquad \text{for all
  } f\in \mathcal D_s(G).
  \]
  Pick now a function $f\in \mathcal D_s(G)$ so that the support of $f$ lies within a fundamental
  domain of the discrete group $\Gamma^{\perp} \subset G$. Then, by \eqref{eq:gabor-frame-walnut},
  \begin{alignat*}{2}
    & &&A \, \Vert f \Vert^2 \le \langle \sum_{\alpha\in\Gamma^{\perp}} M_{s_{\alpha}} T_{\alpha} f,f\rangle \le B \, \Vert f \Vert^2 \\
    &\Leftrightarrow \qquad  &&A \, \Vert f \Vert^2 \le \langle s_{0} f,f\rangle \le B \, \Vert f \Vert^2 \\
    &\Leftrightarrow \qquad &&A \int_G \vert f(x) \vert^2 \, dx \le \int_G \Big( \int_{\Lambda}
    \vert g(x+\lambda) \vert^2 \, d\lambda \Big)\, \vert f(x) \vert^2 \, dx \le B \int_G \vert f(x)
    \vert^2 \, dx.
  \end{alignat*}
  From this assertion (i) follows. By use of \eqref{eq:gabor-frame-walnut-freq}, one proves assertion (ii)
  in the same fashion.
\end{proof}

\subsection{The Janssen representations of the frame operator}
\label{sec:janssen}

The Walnut representation was formulated for semi co-compact Gabor
systems.  In case both $\LL$ and $\LG$ are co-compact, closed
subgroups, we can offer a more time-frequency symmetrical representation of the
Gabor frame operator; this is the so-called Janssen representation.


\begin{theorem} 
\label{th:JanssenRepImpr} 
Let $g,h \in L^2(G)$ and let $\Lambda\subset G,\Gamma\subset \ghat$ be closed, co-compact subgroups such
that $\gaborset$ and $\gaborset[h]$ are Bessel systems. Suppose that the pair $(g,h)$ satisfies
\emph{condition A}:
\begin{equation} 
\label{eq:condA} 
\sum_{\alpha\in \Gamma^{\perp}} \sum_{\beta\in \Lambda^{\perp}}
\absbig{\langle h, E_{\beta} T_{\alpha} g\rangle} < \infty.
\end{equation}
Then
\begin{equation} 
  \label{eq:jannsen_rep}
  S_{g,h} = \sum_{\alpha\in \Gamma^{\perp}}\sum_{\beta\in
    \Lambda^{\perp}} \langle h, E_{\beta}T_{\alpha}g\rangle E_{\beta}T_{\alpha}
\end{equation}
with absolute convergence in the operator norm.
\end{theorem}
\begin{proof}
  Define the operator $\tilde{S}:L^2(G) \to L^2(G)$ by
  \[
  \tilde{S} = \sum_{\alpha\in \Gamma^{\perp}}\sum_{\beta\in \Lambda^{\perp}}
  \innerprod{h}{E_{\beta}T_{\alpha}g} E_{\beta}T_{\alpha}.
  \]
  This series converges absolutely in the operator norm by \eqref{eq:condA}. Hence, the convergence
  is unconditionally. Replacing $s_{\alpha}$ in the Walnut representation by its Fourier series
  representation from Remark \ref{rem:t-alpha-unif-bounded-and-fourier-series} yields
  \begin{align*}
    \langle S_{g,h} f_1,f_2\rangle & = \langle \sum_{\alpha\in \Gamma^{\perp}} M_{s_{\alpha}} 
    T_{\alpha} f_1,f_2\rangle = \langle \sum_{\alpha\in \Gamma^{\perp}} \sum_{\beta\in
      \Lambda^{\perp}} \langle h, E_{\beta} T_{\alpha} g\rangle \beta(x) T_{\alpha} f_1,f_2\rangle
    \\ & = \sum_{\alpha\in \Gamma^{\perp}} \sum_{\beta\in \Lambda^{\perp}}\langle h, E_{\beta}
    T_{\alpha} g\rangle \langle E_{\beta} T_{\alpha} f_1,f_2\rangle = \langle \tilde{S}
    f_1,f_2\rangle
  \end{align*}
  for $f_1,f_2 \in \cD_s$. Since $\cD_s$ is dense in $L^2(G)$, it follows that $S_{g,h}=\tilde{S}$.

\end{proof}

Note that (\ref{eq:jannsen_rep}) indicates convergence in the uniform operator topology, while
Walnut's representation, on the other hand, conveyed convergence in the strong operator topology.

For generators $g,h \in S_0(G)$ in Feichtinger's algebra, the assumptions of the Janssen representation in 
Theorem~\ref{th:JanssenRepImpr} are automatically satisfied. The Bessel condition follows from
Theorem~\ref{thm:feichtinger-implies-bessel}, while \eqref{eq:condA} follows from the
next result.
\begin{proposition} Let $g,h\in S_0(G)$, and let $\Lambda$ and $\Gamma$ be closed subgroups in $G$ and $\ghat$, respectively.
The pair $(g,h)$ satisfies \eqref{eq:condA}, that is,
\[ \int_{\Lambda^{\perp}}\int_{\Gamma^{\perp}} \vert \langle g, E_{\beta}T_{\alpha}h\rangle \vert \, d\alpha \, d\beta < \infty.\]
\end{proposition}
\begin{proof} By \cite[Corollary 7.6.6]{MR1601091} we have that $g,h\in S_0(G)$ implies
  $(x,\omega)\mapsto\langle g, E_{\omega}T_x h \rangle\in S_0(G\times\ghat)$. If we restrict this
  mapping to $\Gamma^{\perp}\times\Lambda^{\perp}\subset G\times \ghat$ and use that $S_0$ is
  continuously embedded into $L^1$, we find that \eqref{eq:condA} is satisfied.
\end{proof}

The next version of the Janssen representation holds for arbitrary (not necessarily co-compact)
closed subgroups $\Lambda\subset G, \Gamma\subset \ghat$. It is called the \emph{fundamental
  identity of Gabor analysis} (FIGA). In  \cite{MR2264211}
Feichtinger and Luef give a detailed answer to when \eqref{eq:jannsen_arb_group_rep} holds in the setting of $\R^n$, see also \cite{MR1601107,MR1601091} for related results. The FIGA was first
proved by Rieffel~\cite{MR941652} for generators $g,h$ in the Schwartz-Bruhat space
$S(G)$. Rieffel's proof uses the Poisson summation formula and also holds for the non-separable case
with closed subgroups in $G\times\ghat$; it is also possbile to give an argument based on Janssen's proof for (lattice) Gabor systems in $L^2(\R)$
 \cite{MR1310658, MR1350700}. 

\begin{theorem} 
\label{th:JanssenRepArbGroup} 
Let $f_1,f_2,g,h \in L^2(G)$, and let $\Lambda\subset G, \Gamma\subset \ghat$ be closed subgroups. Assume that
$\gaborset[g]$ and $\gaborset[h]$ are Bessel systems. If 
\[ (\alpha,\beta)\mapsto
  \langle E_{\overline{\beta}}T_{\alpha} f_1,f_2\rangle
  \langle h, E_{\overline{\beta}} T_{\alpha}g\rangle \in
  L^1(\Gamma^{\perp}\times\Lambda^{\perp}),\] then 
\begin{equation} 
\label{eq:jannsen_arb_group_rep} 
\langle S_{g,h} f_1,f_2\rangle = \int_{\Gamma^{\perp}}\int_{\Lambda^{\perp}} \langle h, E_{\beta}T_{\alpha}g\rangle  \langle E_{\beta}T_{\alpha} f_1,f_2\rangle \, d\beta \, d\alpha.
\end{equation}
\end{theorem}

\section{Co-compact Gabor systems and their adjoint systems}
\label{sec:duality-results}

The Janssen representation shows that the frame operator of a co-compact Gabor system can be written in terms of the system
$\{E_{\beta}T_{\alpha}g\}_{\alpha\in\Gamma^{\perp},\beta\in\Gamma^{\perp}}$. In this section we present further results
that connect a co-compact Gabor system $\gaborset$ with its \emph{adjoint} Gabor system
$\{E_{\beta}T_{\alpha}g\}_{\alpha\in\Gamma^{\perp},\beta\in\Gamma^{\perp}}$.

The time-frequency shifts in a Gabor system and its adjoint system are characterized by the fact
that they commute \cite[Section~3.5.3]{MR1601107},\cite[Lemma 7.4.1]{MR1843717}. That is, for $(\lambda,\gamma)\in \Lambda\!\times\! \Gamma$ the point $(\alpha,\beta)\subset G\times\ghat$ belongs to $\Gamma^{\perp}\!\times\!\Lambda^{\perp}$ if and only if
\[ 
(E_{\gamma}T_{\lambda})( E_{\beta} T_{\alpha}) = (E_{\beta} T_{\alpha}
)(E_{\gamma}T_{\lambda}).  
\] 


We remind the reader our convention equipping the annihilator of $\Lambda$ and $\Gamma$ with the counting measure. The following results will, therefore, only after appropriate modification take the familiar form of the lattice Gabor theory in, e.g., $L^2(\R^n)$.

\subsection{Bessel bound duality}
\label{sec:bessel-bound-duality}

Bessel bound duality states that a co-compact Gabor system is a Bessel system with bound $B$ if, and only
 if, the discrete adjoint Gabor system
 $\{E_{\beta}T_{\alpha}g\}_{\alpha\in\Gamma^{\perp},\beta\in\Lambda^{\perp}}$ is a Bessel system
 with bound $B$. The result is stated in Proposition \ref{thm:bessel-duality}, and its proof is divided into two parts, Lemma \ref{le:bessel-implies-adj-bessel} and
 \ref{le:adj-gab-bessel-implies-gabor-bessel}. 

We begin with the definition of the operator $L_x:D(L_x)\to L^2(\LL)$ with $D(L_x) \subset
\ell^2(\LG^\perp)$. Let $x \in G$, let $g\in L^2(G)$ be given and let
$\{c_{\alpha}\}_{\alpha\in\Gamma^{\perp}}$ be a finite sequence. Then for almost every $x\in G$ we
define the linear operator
\begin{equation} 
  \label{eq:def-pre-gram-Ux} 
  L_x (\{c_{\alpha}\}_{\alpha\in\Gamma^{\perp}}) = \lambda \mapsto  \sum_{\alpha\in\Gamma^{\perp}} g(x-\lambda-\alpha) \, c_{\alpha} , \quad D(L_x)=c_{00}(\LG^\perp).
\end{equation}
Note that $L_x$ essentially (up to complex conjugations, \etc) is the analysis operator, as
introduced in (\ref{eq:analysis-op-fibers}), of the family of fibers associated with the TI system
$\seq{T_\gamma \cF^{-1}T_\lambda g}_{\gamma \in \LG, \lambda \in \LL}$. In light of
Proposition~\ref{thm:TI-gramian-pre-bessel-equi}, we therefore have
the following result.
\begin{lemma} 
  \label{le:pre-gram-Ux-linear-and-bounded} 
  If $\gaborset$ is a Bessel system with bound $B$, then for almost every $x\in G$ the operator
  $L_x$ extends to a linear, bounded operator with domain $\ell^2(\Gamma^{\perp})$ and 
  bound $B^{1/2}$.
\end{lemma}

 Let us now show one direction of the Bessel duality between a
 co-compact Gabor system and its  adjoint.

\begin{lemma}
  \label{le:bessel-implies-adj-bessel} Let $\Lambda\subset G$ and $\Gamma\subset \ghat$ be closed,
  co-compact subgroups. If $\gaborset$ is a Bessel system with bound $B$, then
  $\{E_{\beta}T_{\alpha}g\}_{\alpha\in\Gamma^{\perp},\beta\in\Lambda^{\perp}}$ is a Bessel system
  with bound $B$.
\end{lemma} 
\begin{proof}
  We consider the discrete Gabor system
  $\{E_{\beta}T_{\alpha}g\}_{\alpha\in\Gamma^{\perp},\beta\in\Lambda^{\perp}}$ and its associated
  synthesis mapping
  \[
  \SynthesisOperator :\ell^2(\Gamma^{\perp}\times\Lambda^{\perp})\to L^2(G), \ \ \SynthesisOperator
  c(\alpha,\beta) = \sum_{\alpha\in\Gamma^{\perp}}\sum_{\beta\in\Lambda^{\perp}}c(\alpha,\beta)
  E_{\beta}T_{\alpha}g.
  \]
  We will show that $\SynthesisOperator$ is a well-defined, linear and bounded operator with $\Vert
  \SynthesisOperator \Vert \le B^{1/2}$; the result then follows from
  \cite[Theorem~3.2.3]{MR1946982}. To this end, let $c\in \ell^2(\Gamma^\perp\times\Lambda^{\perp})$ be a finite
  sequence and for each $x\in G$ consider
  \begin{equation} \label{eq:def-m-alpha} m_{\alpha}(x) := \sum_{\beta\in \Lambda^{\perp}}
    c(\alpha,\beta) \beta(x), \ \ \alpha\in\Gamma^\perp.
  \end{equation}
  It is clear that $\{m_{\alpha}(x)\}_{\alpha\in\Gamma^{\perp}}$ is a finite sequence as well.  Note
  that $m_{\alpha}$ as a function of $x\in G$ is constant on cosets of $\Lambda$. Thus $m_{\alpha}$
  defines a function on $G/\Lambda$, which we will denote by $m_{\alpha}(\dot x)$.  By use of the
  identification $G/\Lambda\cong {\widehat{\Lambda^{\perp}}}$ and the Parseval equality, we find
  \begin{align} 
    \int_{G/\Lambda} \vert m_{\alpha}(\dot x) \vert^2 \, d\mu_{G/\Lambda}(\dot x) & = \int_{\widehat{\Lambda^{\perp}}} \Big\vert \sum_{\beta\in\Lambda^{\perp}} c(\alpha,\beta) \beta(x) \, \Big\vert^2 \, d\mu_{\widehat{\Lambda^{\perp}}}(x) \nonumber \\
    & = \Vert
    c(\alpha,\beta) \Vert_{\ell^2(\Lambda^{\perp})}^2 = \sum_{\beta\in\Lambda^{\perp}} \vert
    c(\alpha,\beta)\vert^2 .
    \label{eq:norm-m-lambda}
  \end{align}
  By
  definition we have that
  \[
  \mathcal \SynthesisOperator c=\sum_{\alpha\in\Gamma^{\perp}}\sum_{\beta\in\Lambda^{\perp}}
  c(\alpha,\beta) E_{\beta}T_{\alpha} g = \sum_{\alpha\in\Gamma^{\perp}} M_{m_{\alpha}} T_{\alpha}
  g.
  \] 
  Using this expression together with Weil's formula we find the following for the norm of
  $\SynthesisOperator c$:
  \begin{align} 
    \Vert \SynthesisOperator c \Vert^2 & = \int_G \vert \SynthesisOperator c(x) \vert^2 \, d\mu_{G}(x) =  \int_{G} \sum_{\alpha,\alpha'\in\Gamma^{\perp}} m_{\alpha}(x) g(x-\alpha) \, \overline{m_{\alpha'}(x) g(x-\alpha')} \, d\mu_G(x) \nonumber \\
    & = \int_{G/\Lambda} \int_{\Lambda} \Big( \sum_{\alpha\in\Gamma^{\perp}} m_{\alpha}(\dot x) g(x-\lambda-\alpha) \Big) \Big( \sum_{\alpha'\in\Gamma^{\perp}} \overline{m_{\alpha'}(\dot x) g(x-\lambda-\alpha')} \Big) \, d\mu_{\Lambda}(\lambda) \, d\mu_{G/\Lambda}(\dot x) \nonumber \\
    &  = \int_{G/\Lambda} \Vert L_x m_{\alpha}(\dot x)\Vert^2_{L^2(\Lambda)} \, d\mu_{G/\Lambda}(\dot
    x). \label{eq:adjoint-gabor-Tc-norm}
\end{align}
The rearranging of the summation is possible because the summations over $\Gamma^{\perp}$ are
finite. Since $\gaborset$ is a Bessel system with bound $B$, we know by Lemma
\ref{le:pre-gram-Ux-linear-and-bounded} that $L_x$ is bounded by $B^{1/2}$. We therefore have that
\[ 
\Vert L_x m_{\alpha}(\dot x) \Vert^2 \le B \, \Vert m_{\alpha}(\dot x) \Vert^2 = B
\sum_{\alpha\in\Gamma^{\perp}} \vert m_{\alpha}(\dot x) \vert^2.
\]
Using this together with \eqref{eq:norm-m-lambda} and \eqref{eq:adjoint-gabor-Tc-norm} yields the
following inequality.
\[ 
\Vert\SynthesisOperator c \Vert^2 \le B \int_{G/\Lambda} \sum_{\alpha\in\Gamma^{\perp}} \vert
m_{\alpha}(x)\vert^2 \, d\mu_{G/\Lambda}(\dot x) = B \sum_{\alpha\in\Gamma^{\perp}}
\sum_{\beta\in\Lambda^{\perp}} \vert c(\alpha,\beta) \vert^2 = B \, \Vert c
\Vert_{\ell^2(\Gamma^{\perp}\times\Lambda^{\perp})}^2.
\]
We conclude that $\SynthesisOperator$ is bounded by $B^{1/2}$ and so
$\{E_{\beta}T_{\alpha}g\}_{\alpha\in\Gamma^{\perp},\beta\in\Lambda^{\perp}}$ is a Bessel system
with bound $B$.
\end{proof}

Note that in the classical discrete and co-compact setting we simply apply
Lemma \ref{le:bessel-implies-adj-bessel} to the adjoint Gabor system, as it would also be discrete
and co-compact. However, in our case the Gabor system $\gaborset$ is co-compact and the adjoint
system is discrete (and not necessarily co-compact). We thus need another result for the reverse direction.

In order to prove the reverse direction, Lemma \ref{le:adj-gab-bessel-implies-gabor-bessel}, we will reuse calculations from
Lemma \ref{le:bessel-implies-adj-bessel}. Furthermore, the proof also relies on Lemma \ref{le:lemma22-for-S0-gabor}.
Adapted to co-compact $\Gamma\subset\ghat$ it states that for all $f\in C_c(G)$
\begin{equation} \label{eq:lemma22-for-gabor} \int_{\Gamma} \vert \langle f, E_{\gamma} T_{\lambda} g\rangle \vert^2 \, d\mu_{\Gamma}(\gamma) = \int_G \sum_{\alpha\in\Gamma^{\perp}} f(x) \overline{f(x-\alpha)} \, \overline{T_{\lambda}g(x)} \, T_{\lambda}g(x-\alpha)\, d\mu_G(x).
\end{equation}

\begin{lemma} 
  \label{le:adj-gab-bessel-implies-gabor-bessel}
  Let $\Lambda\subset G$ and $\Gamma\subset \ghat$ be closed, co-compact subgroups. If
  $\{E_{\beta}T_{\alpha}g\}_{\alpha\in\Gamma^{\perp},\beta\in\Lambda^{\perp}}$ is a Bessel system
  with bound $B$, then $\gaborset$ is a Bessel system with bound $B$.
\end{lemma}
\begin{proof} 
  Note that for finite sequences $c\in \ell^2(\Gamma^{\perp}\times\Lambda^{\perp})$ the calculations
  in \eqref{eq:adjoint-gabor-Tc-norm} still hold. We let $m_{\alpha}(x)$ be given as in
  \eqref{eq:def-m-alpha}. By assumption we know that the synthesis mapping $\SynthesisOperator$ of
  the adjoint Gabor system
  $\{E_{\beta}T_{\alpha}g\}_{\alpha\in\Gamma^{\perp},\beta\in\Lambda^{\perp}}$ is bounded by
  $B^{1/2}$. We therefore have that
  \[
  \Vert F c \Vert^2 = \int_{G/\Lambda} \Vert L_x m_{\alpha}(\dot x)\Vert^2_{L^2(\Lambda)} \,
  d\mu_{G/\Lambda}(\dot x) \le B \, \Vert c \Vert_{\ell^2(\Gamma^{\perp}\times\Lambda^{\perp})} \quad \forall c\in \ell^2(\Gamma^{\perp}\times\Lambda^{\perp}).  \]
  By use of \eqref{eq:norm-m-lambda} we rewrite the norm of $c$ and find
  \begin{equation}
    \label{eq:lemma-adj-bessel-implies-bessel-eq} 
    \int_{G/\Lambda} \Vert L_x m_{\alpha}(\dot x)\Vert^2_{L^2(\Lambda)} \, d\mu_{G/\Lambda}(\dot x) \le B \, \int_{G/\Lambda} \Vert m_{\alpha}(\dot x) \Vert_{\ell^2(\Gamma^{\perp})}^2 \, d\mu_{G/\Lambda}(\dot x).
  \end{equation}
  This implies that
  \begin{equation}
    \label{eq:lemma-adj-bessel-implies-bessel-eq2} 
    \Vert L_x m_{\alpha}(\dot x)\Vert^2_{L^2(\Lambda)} \le B \, \Vert m_{\alpha}(\dot x) \Vert_{\ell^2(\Gamma^{\perp})}^2 .
  \end{equation}
  If $c(\alpha,\beta) = 0$ for all $\beta \ne 1$, then $m_{\alpha}(x) = c(\alpha,1)$. Therefore the
  mapping from all finite $c\in \ell^2(\Gamma^{\perp}\times\Lambda^{\perp})$ to $m_{\alpha}(x)$ in
  \eqref{eq:def-m-alpha} is a surjection onto all finite sequences indexed by $\Gamma^{\perp}$. From
  \eqref{eq:lemma-adj-bessel-implies-bessel-eq2} we can therefore conclude that $L_x$ is a bounded
  operator from all finite sequences to $L^2(\Lambda)$ with $\Vert L_x \Vert \le B^{1/2}$. Since
  $L_x$ is also linear, it uniquely extends to a bounded operator from all of
  $\ell^2(\Gamma^{\perp})$ to $L^2(\Lambda)$.

  Let now $f\in C_c(G)$ and consider the finite sequence $c =
  \{\overline{f(x-\alpha)}\}_{\alpha\in\Gamma^{\perp}}$. Replacing $\{m_{\alpha}(\dot
  x)\}_{\alpha\in\Gamma^{\perp}}$ with $c$ in \eqref{eq:lemma-adj-bessel-implies-bessel-eq} yields
  the following inequality:
  \begin{equation} 
    \label{eq:Adj-Gab-Bessel-Implies-Gab-Bessel-Ux-Bessel-bound-eq}
    \int_{G/\Gamma^{\perp}} \Vert L_x c \Vert_{L^2(\Lambda)}^2 \, d\mu_{G/\Gamma^{\perp}}(\dot x)
    \le B \, \int_{G/\Gamma^{\perp}} \sum_{\alpha\in\Gamma^{\perp}} \vert f(x-\alpha)\vert^2 \,
    d\mu_{G/\Gamma^{\perp}}(\dot x) = B \, \Vert f \Vert_{L^2(G)}^2.
  \end{equation} 
  Concerning the left hand side of \eqref{eq:Adj-Gab-Bessel-Implies-Gab-Bessel-Ux-Bessel-bound-eq},
  we find that
  \begin{align}
    &  \int_{G/\Gamma^{\perp}} \Vert L_x c \Vert_{L^2(\Lambda)}^2 \, d\mu_{G/\Gamma^{\perp}}(\dot x) \nonumber \\
    & = \int_{G/\Gamma^{\perp}} \int_{\Lambda} \sum_{\alpha,\alpha'\in\Gamma^{\perp}} g(x-\lambda-\alpha) \, \overline{f(x-\alpha)}\, \overline{g(x-\lambda-\alpha')} \, f(x-\alpha') \, d\lambda \, d\mu_{G/\Gamma^{\perp}} \nonumber \\
    & = \int_{G/\Gamma^{\perp}} \int_{\Lambda} \sum_{\alpha,\alpha'\in\Gamma^{\perp}} g(x-\lambda-\alpha'-\alpha) \, \overline{f(x-\alpha'-\alpha)}\, \overline{g(x-\lambda-\alpha')} \, f(x-\alpha') \, d\lambda \, d\mu_{G/\Gamma^{\perp}} \nonumber \\
    & = \int_{G} \int_{\Lambda} \sum_{\alpha\in\Gamma^{\perp}} g(x-\lambda-\alpha) \, \overline{f(x-\alpha)}\, \overline{g(x-\lambda)} \, f(x) \, d\lambda \, d\mu_{G}(x) \nonumber \\
    & = \int_{\Lambda}\int_{\Gamma} \vert \langle f, E_{\gamma}T_{\lambda} g \rangle \vert^2 \,
    d\mu_{\Gamma}(\gamma) \, d\mu_{\Lambda}(\lambda). \label{eq:from-Ux-to-gabor-bessel-bound}
  \end{align}
  The last equality follows by \eqref{eq:lemma22-for-gabor}. From
  \eqref{eq:Adj-Gab-Bessel-Implies-Gab-Bessel-Ux-Bessel-bound-eq} and
  \eqref{eq:from-Ux-to-gabor-bessel-bound} we conclude that
  \[
  \int_{\Lambda}\int_{\Gamma} \vert \langle f, E_{\gamma}T_{\lambda} g \rangle \vert^2 \,
  d\mu_{\Gamma}(\gamma) \, d\mu_{\Lambda}(\lambda) \le B \, \Vert f \Vert^2 \quad \text{for all } f\in
  C_c(G).
  \] 
  Since this holds for all $f$ in a dense subset of $L^2(G)$ we draw the conclusion that $\gaborset$
  is a Bessel system with bound $B$.
\end{proof}

The combination of Lemma \ref{le:bessel-implies-adj-bessel} and
\ref{le:adj-gab-bessel-implies-gabor-bessel} yields the Bessel bound duality between a co-compact
Gabor system and its discrete adjoint system.

\begin{proposition}
\label{thm:bessel-duality}
Let $B>0$ and $g,h\in L^2(G)$ be given. Let $\Gamma \subset G$ and
$\Lambda \subset \ghat$ be closed, co-compact
subgroups. Then $\gaborset$ is a Bessel system with bound $B$ if, and only if, 
$\{E_{\beta}T_{\alpha}g\}_{\alpha\in\Gamma^{\perp},\beta\in\Lambda^{\perp}}$ is a Bessel system with bound
$B$.
\end{proposition}

\subsection{Wexler-Raz biorthogonality relations}
\label{sec:wexler-raz}

We now turn our attention to a characterization of dual co-compact Gabor frame
generators by a biorthogonality
condition of the corresponding (discrete) adjoint Gabor systems.
Feichtinger and Kozek \cite{MR1601091} proved the Wexler-Raz
biorthogonality relations for Gabor systems with translation and
modulation along \emph{uniform lattices} on elementary LCA groups,
i.e., $G=\R^n \times \T^\ell \times \Z^k \times F_m$, where $F_m$ is a
finite group.
 For a proof in the discrete and finite setting and on the real line we refer to the original papers
\cite{ISI:A1990EJ96700001} and \cite{MR1350700}.

\begin{theorem} 
\label{th:WexRaz} 
Let $\Lambda\subset G$ and $\Gamma\subset \ghat$ be closed, co-compact subgroups. Let
$g,h\in L^2(G)$ and assume that $\gaborset$ and $\gaborset[h]$ are
Bessel systems.  Then the two Gabor systems are dual
frames if, and only if, 
\begin{equation} \label{eq:wexraz} \langle h, E_{\beta}T_{\alpha} g \rangle = \delta_{\beta,1}
  \delta_{\alpha,0} \qquad \forall \alpha\in\Gamma^{\perp},\beta \in \Lambda^{\perp}.
\end{equation}
\end{theorem}
\begin{proof}
  Assume that the two Gabor systems are dual frames. Then, for each $\alpha \in \Gamma^{\perp}$, we
  have $s_\alpha=\delta_{\alpha,0}$ for a.e. $x \in G$.  By uniqueness of the Fourier coefficients
  \eqref{eq:s-Fourier-coeff}, the conclusion in \eqref{eq:wexraz} follows. The converse direction is
  immediate.
\end{proof}

\begin{remark} \label{rem:wexrazeq}(i).
  \begin{enumerate}[(i)] 
    \item From equation \eqref{eq:wexraz} with $\alpha'\in \Gamma^{\perp},
    \beta'\in \Lambda^{\perp}$ we find
    \[
    \delta_{\beta,1}\delta_{\alpha,0} = \langle h,
    E_{\beta}T_{\alpha}g \rangle = \langle E_{\beta'}T_{\alpha'} h,
    \overline{\beta(\alpha)}E_{\beta'\beta} T_{\alpha'+\alpha}g
    \rangle.
    \]
    And thus the Wexler-Raz biorthogonality relations
    \eqref{eq:wexraz} can equivalently be stated as
    \[
    \langle E_{\beta}T_{\alpha}h,E_{\beta'}T_{\alpha'}g\rangle =
    \delta_{\alpha,\alpha'}\delta_{\beta,\beta'} \ \ \forall \,
    \alpha,\alpha' \in\Gamma^{\perp}, \beta,\beta' \in
    \Lambda^{\perp}.
    \]
  \item For canonical dual frames $\gaborset$ and
    $\gaborset[S^{-1}g]$, the biorthogonal sequences
    $\seq{E_{\beta}T_{\alpha}g}_{\alpha \in \LG^\perp, \beta \in
      \LL^\perp}$ and $\seq{E_{\beta}T_{\alpha}S^{-1}g}_{\alpha \in
      \LG^\perp, \beta \in \LL^\perp}$ are actually dual Riesz bases
    for the subspace
    $\overline{\Span}\seq{E_{\beta}T_{\alpha}g}_{\alpha \in \LG^\perp,
      \beta \in \LL^\perp}$, see \cite[Proposition 3.3]{MR1350700}.
  \end{enumerate}
\end{remark}

\subsection{The duality principle}
\label{sec:duality-principle}

The duality principle for lattice Gabor systems in $L^2(\R^n)$ was proven simultaneously by three groups of authors,
Daubechies, Landau and Landau~\cite{MR1350701}, Janssen~\cite{MR1350700}, and Ron and
Shen~\cite{MR1460623}. Theorem~\ref{thm:ron-shen-duality} below
generalizes this principle to co-compact Gabor systems in $L^2(G)$. 
Our proof of the duality principle relies on the following result on Riesz sequences in abstract Hilbert spaces, \cf
Definition~\ref{def:riesz-seq}.  It is a subspace variant of \cite[Theorem
3.4.4]{MR2428338} and \cite[Theorem 7.13]{MR2744776}; its proof is due to Ole
Christensen.

\begin{theorem}
\label{thm:riesz-seq}
  Let $\seq{f_k}$ be a sequence in a Hilbert space. Then the following statements are equivalent:
  \begin{enumerate}[(a)]
  \item $\seq{f_k}$ is a Riesz sequence with lower bound $A$ and upper bound $B$, \label{item:riesz1}
  \item $\seq{f_k}$ is a Bessel system with bound $B$ and possesses a biorthogonal system
    $\seq{g_k}$ that is also a Bessel system with bound $A^{-1}$. \label{item:riesz2}
  \end{enumerate}
\end{theorem}
\begin{proof}
  Assume that \eqref{item:riesz1} holds. Set $V=\overline{\Span}\seq{f_k}$. Let $\seq{g_k}$ be the
  unique dual Riesz sequence of $\seq{f_k}$ in $V$ so that $\overline{\Span}\seq{g_k}=V$. This
  implies \eqref{item:riesz2}.

  Assume that \eqref{item:riesz2} holds. Since $\seq{f_k}$ and $\seq{g_k}$ are biorthogonal, it
  follows that
  \[
  f_j = \sum_k \innerprod{f_j}{g_k}f_k
  \]
  for all $j$. By linearity, we have, for any $f \in \Span\seq{f_k}$,
  \[
  f = \sum_k \innerprod{f}{g_k}f_k.
  \]
  This formula extends to $\overline{\Span}\seq{f_k}$ by continuity. Now, for any $f \in
  \overline{\Span}\seq{f_k}$, we have
\begin{align}
  \norm{f}^2 &= \abs{\innerprod{f}{f}} = \absBig{\sum_k\innerprod{f}{g_k}\innerprod{f_k}{f}} \nonumber \\
  &\le \biggl(\sum_k\abs{\innerprod{f}{g_k}}^2 \, \sum_k\abs{\innerprod{f}{f_k}}^2 \biggr)^{1/2} \nonumber \\
  &\le A^{-1/2} \norm{f} \biggl(
  \sum_k\abs{\innerprod{f}{f_k}}^2\biggr)^{1/2}. \label{eq:bessel-and-dual}
\end{align}
We see that $\seq{f_k}$ is a frame sequence with lower frame bound $A$; by assumption the upper
frame bound is $B$. By the fact that $\seq{f_k}$ possesses a biorthogonal sequence, it follows that
$\seq{f_k}$ is, in fact, a Riesz sequence with the same bounds.
\end{proof}


\begin{theorem}
  \label{thm:ron-shen-duality}
  Let $g \in L^2(G)$. Let $\Lambda\subset G$ and $\Gamma\subset \ghat$
  be closed, co-compact subgroups. Then $\gaborset$
  is a frame for $L^2(G)$ with bounds $A$ and $B$ if, and only if, 
  $\{E_{\beta}T_{\alpha}g\}_{\alpha\in\Gamma^{\perp},\beta\in\Lambda^{\perp}}$ Riesz sequence with
  bounds $A$ and $B$.
\end{theorem}
\begin{proof}
  Let $\gaborset$ be a frame with bounds $A$ and $B$. The canonical dual frame $\gaborset[S^{-1}g]$
  has bounds $B^{-1}$ and $A^{-1}$. By Proposition~\ref{thm:bessel-duality}, the sequences
  $\{E_{\beta}T_{\alpha}g\}_{\alpha\in\Gamma^{\perp},\beta\in\Lambda^{\perp}}$ and
  $\{E_{\beta}T_{\alpha}S^{-1}g\}_{\alpha\in\Gamma^{\perp},\beta\in\Lambda^{\perp}}$ are Bessel systems with
  bound $B$ and $A^{-1}$, respectively. By Wexler-Raz biorthogonal relations, these two families are
  biorthogonal, hence, by Theorem~\ref{thm:riesz-seq},
  $\{E_{\beta}T_{\alpha}g\}_{\alpha\in\Gamma^{\perp},\beta\in\Lambda^{\perp}}$ is a Riesz sequence
  with bounds $A$ and $B$.

  Conversely, suppose $\{E_{\beta}T_{\alpha}g\}_{\alpha\in\Gamma^{\perp},\beta\in\Lambda^{\perp}}$
  is a Riesz sequence with bounds $A$ and $B$. The dual Riesz sequence of
  $\{E_{\beta}T_{\alpha}g\}_{\alpha\in\Gamma^{\perp},\beta\in\Lambda^{\perp}}$ is of the form
  $\{E_{\beta}T_{\alpha}h\}_{\alpha\in\Gamma^{\perp},\beta\in\Lambda^{\perp}}$ for some $h \in
  L^2(G)$ and has bounds $B^{-1}$ and $A^{-1}$. Using Proposition~\ref{thm:bessel-duality} we see
  that $\gaborset$ has Bessel bound $B$. On the other hand, $\gaborset[h]$ has Bessel bound
  $A^{-1}$. By Wexler-Raz biorthogonal relations, $\gaborset$ and $\gaborset[h]$ are dual frames. By
  a computation as in (\ref{eq:bessel-and-dual}), we see that $A$ is a lower frame bound for
  $\gaborset$.
\end{proof}

The co-compactness assumption on $\LL$ and $\LG$ is a natural framework
for the duality principle. Indeed, if the Gabor system is not co-compact,
the adjoint system is not discrete. However, we know by a result of
Bownik and Ross~\cite{BowRos2014} that continuous Riesz sequences do
not exist. Hence, if either $\LL$ or $\LG$ is not co-compact, the
adjoint Gabor system cannot be a Riesz ``sequence''.

Since a Riesz sequence with bounds $A=B$ is an orthogonal sequence, we have the following corollary
of Theorem~\ref{thm:ron-shen-duality}.
\begin{corollary}
  Let $\Gamma$ and $\Lambda$ be closed, co-compact subgroups. A Gabor system
  $\{E_{\gamma}T_{\lambda}g\}_{\lambda\in\Lambda,\gamma\in\Gamma}$ is a tight frame if, and only if, 
  $\{E_{\beta}T_{\alpha}g\}_{\alpha\in\Gamma^{\perp},\beta\in\Lambda^{\perp}}$ is an orthogonal
  system. In these cases, the frame bound is given by $A=\Vert g \Vert^2.$
\end{corollary}

We end this paper with the following general remark:
\begin{remark}
  We have stated the results of the current paper for Gabor systems
  generated by a single function, however, most of the results can be
  stated for finitely or even infinitely many generators; the
  non-existence result, Theorem
  \ref{thm:non-existence-cont-gabor-critical-sampling}, is of course
  an exception to this rule.
\end{remark}

\smallskip
\paragraph{Acknowledgments.} The authors thank Ole Christensen for
useful discussions and for the proof of
Theorem~\ref{thm:riesz-seq}. The first-named author also thanks Hans G.\ Feichtinger for discussions and pointing out references concerning $S_0$. 


\def\cprime{$'$} \def\cprime{$'$} \def\cprime{$'$}
  \def\uarc#1{\ifmmode{\lineskiplimit=0pt\oalign{$#1$\crcr
  \hidewidth\setbox0=\hbox{\lower1ex\hbox{{\rm\char"15}}}\dp0=0pt
  \box0\hidewidth}}\else{\lineskiplimit=0pt\oalign{#1\crcr
  \hidewidth\setbox0=\hbox{\lower1ex\hbox{{\rm\char"15}}}\dp0=0pt
  \box0\hidewidth}}\relax\fi} \def\cprime{$'$} \def\cprime{$'$}
  \def\cprime{$'$} \def\cprime{$'$} \def\cprime{$'$} \def\cprime{$'$}

\end{document}